\documentclass[11 pt, reqno, twoside, pdftex]{amsart}

\usepackage[latin1]{inputenc}
\usepackage{lmodern}
\usepackage[T1]{fontenc}
\usepackage[english]{babel}
\usepackage{ngerman}
\usepackage{amsmath}
\usepackage{amssymb}
\usepackage{amsfonts}
\usepackage{amstext}
\usepackage{amsthm}
\usepackage{hyperref}
\usepackage{xcolor}
\usepackage{graphicx}
\usepackage{subfigure}

\usepackage[a4paper, hdivide={3.0 cm,,3.0cm},vdivide={3.5cm,,4.2cm}]{geometry}

\allowdisplaybreaks

\begin{document}

\theoremstyle{plain}
	\newtheorem{Pp}{Proposition}[section]
	\newtheorem{Thm}[Pp]{Theorem}
	\newtheorem{Lm}[Pp]{Lemma}
	\newtheorem{Cor}[Pp]{Corollary}
\theoremstyle{definition}
	\newtheorem{Df}[Pp]{Definition}
	\newtheorem{Cond}[Pp]{Condition}
	\newtheorem{Ass}[Pp]{Assumption}
	\newtheorem{Rm}[Pp]{Remark}
	\newtheorem{Emp}[Pp]{}

\title[Geometric fiber lay-down and convergence to equilibrium]{Geometry, mixing properties and hypocoercivity of a degenerate diffusion arising in technical textile industry}

\author[M.~Grothaus]{Martin Grothaus}
\address{
Martin Grothaus, Mathematics Department, University of Kaiserslautern, \newline 
P.O.Box 3049, 67653 Kaiserslautern, Germany.\newline 
{\rm \texttt{Email:~grothaus@mathematik.uni-kl.de}},\newline
Functional Analysis and Stochastic Analysis Group, \newline
{\rm \texttt{URL:~http://www.mathematik.uni-kl.de/$\sim$wwwfktn/ }}} 

\author[A.~Klar]{Axel Klar}
\address{
Axel Klar, Mathematics Department, University of Kaiserslautern, \newline 
P.O.Box 3049, 67653 Kaiserslautern, Germany.\newline 
{\rm \texttt{Email:~klar@mathematik.uni-kl.de}},\newline
Technomathematics Group, \newline
{\rm \texttt{URL:~http://wwwagtm.mathematik.uni-kl.de/agtm/ }}} 

\author[J.~Maringer]{Johannes Maringer}
\address{
Johannes Maringer, Mathematics Department, University of Kaiserslautern, \newline 
P.O.Box 3049, 67653 Kaiserslautern, Germany.\newline 
{\rm \texttt{Email:~maringer@itwm.fraunhofer.de}},\newline
Technomathematics Group, \newline
{\rm \texttt{URL:~http://wwwagtm.mathematik.uni-kl.de/agtm/ }}} 

\author[P.~Stilgenbauer]{Patrik Stilgenbauer}
\address{
Patrik Stilgenbauer, Mathematics Department, University of Kaiserslautern, \newline
P.O.Box 3049, 67653 Kaiserslautern, Germany. \newline
{\rm \texttt{Email:~stilgenb@mathematik.uni-kl.de}}, \newline
Functional Analysis and Stochastic Analysis Group, \newline
{\rm \texttt{URL:~http://www.mathematik.uni-kl.de/$\sim$wwwfktn/ }}}

\date{\today}



\begin{abstract}
We study a stochastic equation modeling the lay-down of fibers in the production process of nonwovens. The equation can be formulated as some manifold-valued Stratonovich stochastic differential equation living on $\mathbb{R}^d \times \mathbb{S}^{d-1}$, $d \geq 2$. Especially, we study the long time behaviour of the stochastic process. Demanding mathematical difficulties arising due to the degeneracity of the lay-down equation and its associated generator. We prove strong mixing properties by making use of the hypoellipticity of the generator and a new version of Doob's theorem derived recently in \cite{GN12}. Moreover, we show convergence to equilibrium exponentially fast with explicitly computable rate of convergence. This analytic approach uses powerful modern Hilbert space methods from the theory of hypocoercivity developed in \cite{DMS11}. Summarizing, we give interesting mathematical applications of geometric stochastic analysis to real world problems.
\end{abstract}

\maketitle

\section{Introduction} \label{Introduction}

This article is about the mathematical analysis of new fiber lay-down equations developed recently in \cite{KMW12}. Herein fiber lay-down processes arise in the production process of nonwovens and the expression is used for the description of the forms generated by the stochastic lay-down of flexible fibers. The understanding, optimization and mathematical simulation of such fiber webs is of great industrial interest, we refer to \cite{KMW09} and references therein. Areas of application where these stochastic lay-down processes can be observed include e.g.~composite materials (filters), textiles, as well as the hygiene industry. In \cite{MW06} a general mathematical model describing the full fiber spinning process is introduced and is nowadays implemented in the software tool FYDIST developed at the Fraunhofer ITWM, Kaiserslautern. Nevertheless, the numerical simulation leads to excessively large computation times. Hence it would be desirable to have simplified stochastic models at hand simulating a virtual fiber web in a fast and efficient way. The first of such surrogate models, the basic two-dimensional model, is developed in \cite{GKMW07} and reads as some It\^{o} stochastic differential equation (SDE) in $\mathbb{R}^3$ of the form
\begin{align} \label{eq_2D_fiber_lay_down} 
&\mathrm{d}\xi_t = \tau(\alpha_t)\, \mathrm{dt}\\
&\mathrm{d}\alpha_t = - \nabla \Phi (\xi_t) \cdot \tau^\bot(\alpha_t)\, \mathrm{dt} + \sigma \, \mathrm{d}W_t. \nonumber
\end{align}
Here $\tau(\alpha)=(\cos(\alpha),\sin(\alpha))^T$, $\tau^\bot = \frac{\partial \tau}{\partial \alpha}$, $\sigma$ is a nonnegative constant and $W$ is a standard one-dimensional Brownian motion. $\Phi:\mathbb{R}^2 \rightarrow \mathbb{R}$ is a suitable function called the potential. This basic model has been extended in \cite{KMW12} to the more realistic three-dimensional case which serves as starting point for all our mathematical studies in the underlying article. And exactly this extension to the three-dimensional case requires the usage of a differentialgeometric language.
\medskip

At first we discuss the geometry underlying the previously mentioned equation introduced in \cite{KMW12}. Starting with the two-dimensional model \eqref{eq_2D_fiber_lay_down}  we present a new differential geometric derivation and formulation of the model from \cite{KMW12} in each dimension $d \in \mathbb{N}$, $d \geq 2$. Therefore we formulate Equation \eqref{eq_2D_fiber_lay_down} first in its most natural way on $\mathbb{R}^2 \times \mathbb{R} / {2 \pi \mathbb{Z}}$. The latter space is diffeomorphic to $\mathbb{R}^2 \times \mathbb{S}^1$ where $\mathbb{S}$ stands as usual for the unit sphere. Consequently, we get the analogue equation of \eqref{eq_2D_fiber_lay_down} living on $\mathbb{R}^2 \times \mathbb{S}^1$. The resulting equation can then directly be translated to higher dimensions and yields the final $d$-dimensional fiber lay-down model given as some manifold-valued Stratonovich SDE with state space $\mathbb{R}^d \times \mathbb{S}^{d-1}$ by
\begin{align} \label{Fiber_Model_Intro}
&\mathrm{d}\xi_t = v_t \, \mathrm{dt} \\
&\mathrm{d}v_t = - (I- v_t \otimes v_t) \nabla \Phi (\xi_t) \, \mathrm{dt} + \sigma\,(I-v_t \otimes v_t) \circ \mathrm{d}W_t. \nonumber
\end{align}
Here $x \otimes y = x y^T$, $\Phi \in C^\infty(\mathbb{R}^d)$ and $W$ is a standard $d$-dimensional Brownian motion, see Section \ref{Differentialgeometric_modeling} for details. This precise fomulation is also necessary for all the forthcoming analysis and we strongly believe that the geometric derivation can be helpful for every applied mathematician who aims to derive similiar manifold-valued stochastic equations. Furthermore, we provide concret numerical simulation formulas of \eqref{Fiber_Model_Intro} in specific local coordinate systems, so called local SDEs. As consequence, for $d=2$ we obtain back \eqref{eq_2D_fiber_lay_down} and for $d=3$ the local SDE reduces to the three-dimensional model derived in \cite{GKMW07}. We finish the geometric discussion in Section \ref{solutions_SDEs}. Therein we further discuss basic existence statements of all occuring stochastic equations by constructing global solutions to the basic two-dimensional model \eqref{eq_2D_fiber_lay_down} with state space $\mathbb{R}^3$ and the general geometric fiber lay-down model \eqref{Fiber_Model_Intro} living on $\mathbb{R}^d \times \mathbb{S}^{d-1}$. 
\medskip

Of course, the cases $d=2$ and $d=3$ are the physical relevant ones. Nevertheless, for all the forthcoming analysis, it is elegant to study Equation \eqref{Fiber_Model_Intro} in its general mathematical form. We emphasize that our upcoming stochastic and functionalanalytic considerations of the $d$-dimensional fiber lay-down model are done in a coordinate free way.
\medskip

The density of the stochastic process solving \eqref{Fiber_Model_Intro} satisfes the associated Fokker-Planck evolution equation. An important criterion for the quality of the fiber web and the resulting nonwoven material is how fast the process converges towards its stationary state. Of essential interest is therefore the speed of convergence. Moreover, from a practical point of view, process parameters should be adjusted in order to obtain optimal convergence to equilbrium.
\medskip

This motivates that we devote the main interest in the underlying paper to the study of the long-time behaviour of the $d$-dimensional fiber lay-down equation. Demanding mathematical difficulties are then occuring due to the degeneracity of \eqref{Fiber_Model_Intro}. The convergence to equilibrium of the \glqq flat\grqq~two-dimensional fiber lay-down model is already analyzed in several articles. Basing upon the theory of Dirichlet forms and operator semigroups, the approach of \cite{GK08} gives an ergodic theorem and establishes explicit rates of convergence. The underlying object of study therein is the two-dimensional fiber lay-down Kolmogorov PDE and its corresponding Kolmogorov operator. Another approach is presented in \cite{DKMS11}. There the authors analyze the two-dimensional fiber lay-down Fokker-Planck PDE together with its associated generator and prove convergence to equilibrium with an exponential, explicitly computable rate of convergence. This approach uses modern methods from \cite{DMS11} in which a new hypocoercivity theory in a Hilbert space setting is developed. For a general study of the theory of hypocoercivity, the reader may consult \cite{Vil09}. Finally, the quite recently published article \cite{KSW11} uses a probabilistic approach. The authors are able to derive strong mixing properties under weak assumptions on the potential and, assuming some stronger conditions, they can prove geometric ergodicity of the two-dimensional fiber lay-down process.
\medskip

In Section \ref{section_mixing_properties} we start studying the long time behaviour of our general, geometric $d$-dimensional fiber lay-down process  by proving strong mixing properties, see Theorem \ref{Thm_strong_mixing}. The approach is disjoint and different to the one given in the two-dimensional setting in \cite{KSW11}, and moreover, it has even been developed at the same time independently. Our approach works for smooth potentials $\Phi$. This is a slightly stronger assumption as made in \cite{KSW11} in the two-dimensional case. Nevertheless, our strategy applies to the full manifold-valued $d$-dimensional fiber lay-down SDE containing especially the two-dimensional situation. We make use of a new version of Doob's theorem, see \cite{GN12}, which perfectly fits into the fiber lay-down setting.
\medskip

Afterwards, we switch to functional analysis and apply the fascinating and powerful hypocoercivity theory from \cite{DMS11} once more, see Section \ref{section_Hypocoercivity}. We generalize the strategy in \cite{DKMS11} to the $d$-dimensional setting requiring some differential geometric tools. The object of interest in this section is the hypoelliptic Kolmogorov operator associated to \eqref{Fiber_Model_Intro} given by
\begin{align*}
L&= v \cdot \nabla_\xi - (I-v \otimes v) \nabla \Phi(\xi) \cdot \nabla_v + \frac{\sigma^2}{2}  \, \Delta_{\mathbb{S}^{d-1}}
\end{align*}
which is analyzed in an appropriately chosen $L^2$-space. In particular, we make use of modern entropy methods. We obtain convergence to equilibrium exponentially fast with explicitly computable rate of convergence, see Theorem \ref{Thm_Hypocoercivity}.
\medskip

Moreover, in Section \ref{Setup_and_Notations} we introduce some basic geometric language and give a short introduction to the concept of manifold-valued Stratonovich SDEs. Further useful statements needed for our analysis are proven in the Appendix.
\medskip

Finally, the progress achieved in this paper may be summarized by the following list of main results:
\begin{itemize}
\item
Differential geometric derivation and formulation of the fiber lay-down model in each dimension as some manifold-valued Stratonovich SDE.
\item
Construction of strong solutions to all occuring fiber lay-down equations.
\item
Strong mixing properties of the geometric fiber lay-down process under weak assumptions on the potential, see Theorem \ref{Thm_strong_mixing}, by making use of a new version of Doob's theorem proven in \cite{GN12}.
\item
Exponential convergence to equilibrium of the geometric fiber lay-down process towards a unique stationary state with explicitly computable rate of convergence, see Theorem \ref{Thm_Hypocoercivity}. Here we use modern Hilbert space methods from the theory of hypocoercivity developed in \cite{DMS11}.
\end{itemize}

\section{Setup and Notations} \label{Setup_and_Notations}

Before starting we introduce some notations, fix the language concerning the $d$-sphere and give a short introduction to the concept of Stratonovich stochastic differential equations (SDEs) on manifolds. Until the end of this article we follow the notations and language introduced in the underlying section without further mention this again. 
\medskip

$C^\infty(\mathbb{X})$ denotes the set of all infinitely often differentiable functions $f: \mathbb{X} \rightarrow \mathbb{R}$ on some differentiable manifold $\mathbb{X}$. The index $c$ indicates compact support.  $\nabla$ (or also denoted by $\nabla_x$ or $\nabla_{\mathbb{R}^d}$) always denotes the usual gradient operator in $\mathbb{R}^d$, $d \in \mathbb{N}$, (with respect to the variable $x$)  as column vector. $| \cdot|$ is the standard euclidean norm. The standard euclidean scalar product is simply denoted by $\cdot$ or also by $\left( \cdot, \cdot \right)_{\text{euc}}$. Superscript $T$ denotes the transpose of some matrix. The expression smooth means that the underlying object is of class $C^\infty$. $I$ is the identity matrix. Partial derivatives in $\mathbb{R}^d$ with respect to some variable $x$ are denoted as usual by $\frac{\partial}{\partial x}$ or for short by $\partial_x$. Convention: Any vector $x \in \mathbb{R}^d$ is always understood as column vector. And the notation $(x,y)$ for $x \in \mathbb{R}^d$, $y \in \mathbb{R}^k$, is understood as column vector.
\medskip

Next let us fix some geometric language. Let $d \in \mathbb{N}$. We consider the $d$-sphere given by  $\mathbb{S}^d= \{ v \in \mathbb{R}^{d+1}~|~|v|^2 =1 \}$. The algebraic tangent space $T_v\mathbb{S}^d$ at the point $v \in \mathbb{S}^d$ can naturally be embedded into $\mathbb{R}^{d+1}$. Under this identification any $\mathbb{R}$-derivation $D \in T_v\mathbb{S}^d$, $v \in \mathbb{S}^d$, corresponds to some $A \in \mathbb{R}^{d+1}$ with $ \left( A , v \right)_{\text{euc}} =0$. We canonically identify $D$ and $A$, in notation $D \equiv A$. For $f \in C^\infty(\mathbb{S}^d)$ it holds $Df=A \cdot \nabla_{\mathbb{R}^{d+1}} \widetilde{f}(v)$ where $\widetilde{f}$ is any smooth extension of $f$ defined in an open neighbourhood in $\mathbb{R}^{d+1}$ of $v$. So this justifies the notation $Af=A \cdot \nabla_v f(v)$ instead of $Df$. Moreover, for a given smooth vector field $\mathcal{A}$ on $\mathbb{S}^d$ defined by $\mathbb{S}^d \ni v \mapsto \mathcal{A}(v) \in T_v\mathbb{S}^d$, we write $\mathcal{A}f(v)=\mathcal{A}(v)f$, $f \in C^\infty(\mathbb{S}^d)$. The spherical gradient of such an $f$ is denoted by $\text{grad}_{\mathbb{S}^d}f$. One has 
\begin{align*}
\text{grad}_{\mathbb{S}^d}f(v)= (I-v \otimes v) \,\nabla_{\mathbb{R}^{d+1}} \widetilde{f} (v),~f \in C^\infty(\mathbb{S}^d),~v \in \mathbb{S}^d.
\end{align*}
Here $x \otimes y:=x y^T$, $x,y \in \mathbb{R}^{d+1}$ and $\widetilde{f}$ is chosen as before. This definition is again independent of the local smooth extension $\widetilde{f}$ for $f$. In short notation we also write  $\text{grad}_{\mathbb{S}^d}f(v)= (I-v \otimes v) \,\nabla_{v} f(v)$ for $f \in C^\infty(\mathbb{S}^d)$ and $v \in \mathbb{S}^d$.

\subsection{Stratonovich SDEs on manifolds.} Now let $\mathbb{X}$ be a $C^\infty$-manifold, assumed to be second-countable and Hausdorff. $\hat{\mathbb{X}}=\mathbb{X} \cup \{ \Delta \}$ is the one-point compactification. Our main reference is \cite{Hsu02}. Consider also the excellent german book \cite{HT94}. 
\medskip

\noindent \textit{Solution concept}. See \cite[Def.~7.41]{HT94} or \cite[Ch.~V]{IW89}. Let $(\Omega, \mathcal{F}, \mathbb{P}, \{\mathcal{F}_t\}_{t \geq 0})$ be a standard filtered probability space equipped with an $r$-dimensional standard $\{\mathcal{F}_t\}_{t \geq 0}$\,-\,Brownian motion $W=\{W_t\}_{t \geq 0}$. Let $\mathcal{V}_0,~\mathcal{V}_1, \ldots, \mathcal{V}_r$ be smooth vector fields on $\mathbb{X}$ and let $x_0: \Omega \rightarrow \mathbb{X}$ be $\mathcal{F}_0$-measurable. A solution $X=\{X_t\}_{t \geq 0}$ of the Stratonovich stochastic differential equation
\begin{align} \label{Df_Stratonovich_equation}
\mathrm{d}X_t = \mathcal{V}_0 (X_t) \,\mathrm{dt} + \sum_{j=1}^r \mathcal{V}_j (X_t) \circ \mathrm{d}W_t^{(j)}
\end{align}
with initial condition $X_0=x_0$ is any $\{\mathcal{F}_t\}_{t \geq 0}$\,-\,adapted, continuous process on $\hat{\mathbb{X}}$ having $\Delta$ as a trap such that the following is satisfied: $X_0=x_0$ $\mathbb{P}${-a.s.}~and for every $f \in C_{c}^\infty(\mathbb{X})$, 
\begin{align*}
f(X_t) - f(X_0) = \int_0^t (\mathcal{V}_0 f) (X_s) \, \mathrm{d}s + \sum_{j=1}^r  \int_0^t (\mathcal{V}_j f) (X_s) \circ \mathrm{d}W_s^{(j)},~t \geq 0.
\end{align*}
We call $e(X):= \inf_{t \geq 0} \{ X_t = \Delta \}$ the explosion time or lifetime of $X$. Herein $f(\Delta):=0$ for $f \in C_{c}^\infty(\mathbb{X})$. Note that the definition of a solution $X$ is always understood relative to $(\Omega, \mathcal{F}, \mathbb{P}, \{\mathcal{F}_t\}_{t \geq 0}, W)$. In short form we write: $X$ is a solution to $\text{SDE}(\mathcal{V}_0,\ldots,\mathcal{V}_r;W,x_0)$. Finally, consider \cite[Def.~1.2.3]{Hsu02} for an equivalent definition involving test functions from $C^\infty(\mathbb{X})$.
\medskip

\noindent \textit{State space transform}. See \cite[Prop.~1.2.4]{Hsu02}. Let $\mathbb{X}$ and $\widetilde{\mathbb{X}}$ be diffeomorphic with diffeomorphism $\varphi:\mathbb{X} \rightarrow \widetilde{\mathbb{X}}$. Then $X=\{ X_t \}_{t \geq 0}$ solves SDE \eqref{Df_Stratonovich_equation} iff $\widetilde{X}=\{\widetilde{X}_t\}_{ t\geq 0}$, where $\widetilde{X}_t:= \varphi(X_t)$, solves its associated SDE on $\widetilde{\mathbb{X}}$ in which the $\mathcal{V}_j$ are replaced by their corresponding pushforward vector fields $\widetilde{\mathcal{V}_j}$ on $\widetilde{\mathbb{X}}$. Clearly it holds $(\widetilde{\mathcal{V}_j} \widetilde{f})(\widetilde{p})=\mathcal{V}_jf(p)$ where $\widetilde{f} \circ \varphi =f$, $\widetilde{p}= \varphi(p)$ for $f \in C^\infty(\mathbb{X}),~p \in \mathbb{X}$. The SDEs on $\mathbb{X}$ and $\widetilde{\mathbb{X}}$ are called equivalent.
\medskip

\noindent \textit{Generator of a Stratonovich SDE.} See \cite[Sec.~1.3]{Hsu02}. Any solution $X=\{ X_t \}_{t \geq 0}$ of the Stratonovich SDE \eqref{Df_Stratonovich_equation} is an $L$-diffusion process generated by the second order H"{o}rmander type operator $L:C^\infty(\mathbb{X}) \rightarrow C^\infty(\mathbb{X})$ given by $L=\mathcal{V}_0 + \frac{1}{2} \sum_{j=1}^r \mathcal{V}_j^2$. Two such solutions $X$, $Y$ (possibly defined on different probability spaces) of \eqref{Df_Stratonovich_equation} with the same initial law weakly coincide, i.e., they induce the same law on the path space $W(\mathbb{X})$. Here $W(\mathbb{X})$ is the set of all continuous paths $\omega:[0,\infty) \rightarrow \hat{\mathbb{X}}$ satisfying $\omega(s)=\Delta$ for all $s \geq t$ whenever $\omega(t)=\Delta$. $W(\mathbb{X})$ is equipped with the canonical filtration.
\medskip

\noindent \textit{Brownian motion on Riemannian manifolds.} See \cite[Sec.~3.2]{Hsu02}. In case $(\mathbb{M},g)$ is a Riemannian manifold, any $\frac{1}{2} \Delta_\mathbb{M}$-diffusion process is called a Brownian motion on $\mathbb{M}$. Thus a Brownian motion can be generated by a Stratonovich SDE whose generator is equal to $\frac{1}{2} \Delta_\mathbb{M}$. Here $\Delta_\mathbb{M}$ denotes the Laplace-Beltrami on $(\mathbb{M},g)$.

\section{The fiber lay-down geometry} \label{Differentialgeometric_modeling}

The three-dimensional fiber lay-down model is already presented in \cite{KMW12}. Consider the latter for further interpretation. In this section we give a second view by view. As described in the introduction we now present a new differentialgeometric derivation and formulation of the model in each dimension motivated by the original ideas from \cite{KMW12}. By the way, this shows up a fascinating interaction between applied and pure mathematics and illustrates the highly geometric nature of our model. We make use of the notations and the geometric language introduced in Section \ref{Setup_and_Notations}. For further background information in stochastic geometry, we again refer to \cite{Hsu02}. 

\subsection{The two-dimensional fiber lay-down SDE on \texorpdfstring{$\mathbb{R}^3$}{}} 
We start with the basic two-dimensional model from \cite{GKMW07} which describes the lay-down of a single fiber as a curve $\xi_t,~t \geq 0,$ in the two-dimensional plane. The model is formulated in the introduction, see Equation \eqref{eq_2D_fiber_lay_down}. Therein, note that $\sigma  \, \mathrm{d}W_t = \sigma  \circ \mathrm{d}W_t$. Here $\circ$ signifies the Stratonovich integral.
\medskip

Let $\widetilde{X}=\{\widetilde{X}_t\}_{t \geq 0}$ be the strong solution of SDE \eqref{eq_2D_fiber_lay_down} associated to some fixed standard one-dimensional Brownian motion $W=\{W_t\}_{t \geq 0}$. We use the notation $\widetilde{X}_t=(\xi_t, \widetilde{\alpha_t})$, $t \geq 0$. In Section \ref{solutions_SDEs} we will see that a strong solution solution exists whenever $\Phi \in C^\infty(\mathbb{R}^2)$ although the drift vector of the SDE is not globally Lipschitz continuous in this case. 

\subsection{The two-dimensional model on \texorpdfstring{$\mathbb{R}^2 \times \mathbb{T}$}{}} \label{2d_model_torus} 
The third component of the SDE \eqref{eq_2D_fiber_lay_down} has the interpretation of being an angle. Henceforth, it is natural to formulate the previous SDE on $\mathbb{R}^2 \times \mathbb{T}$ with $\mathbb{T}:=\mathbb{R} / {2 \pi \mathbb{Z}}$. This can be done as follows. Consider the canonical projection $\mathrm{P}: \mathbb{R}^3 \rightarrow \mathbb{R}^2 \times \mathbb{T}$ mapping $(\xi,\widetilde{\alpha})$ to $(\xi,\alpha)$ where $\alpha:=[\widetilde{\alpha}]$. Then we introduce the process $X=\{X_t\}_{t \geq 0}$ defined by $X_t:= \mathrm{P}(\widetilde{X}_t)$, $t \geq 0$. Herein $\widetilde{X}$ is the strong solution associated to \eqref{eq_2D_fiber_lay_down} defined previously. We write $X_t=(\xi_t, \alpha_t)$, $\alpha_t = [\widetilde{\alpha_t}]$. We claim that $X$ solves the Stratonovich SDE 
\begin{align} \label{eq_2D_fiber_lay_down_manifold}
\mathrm{d}X_t = \mathcal{A}_0 (X_t) \,\mathrm{dt} + \mathcal{A}_1 (X_t) \circ \mathrm{d}W_t
\end{align} 
on the manifold $\mathbb{R}^2 \times \mathbb{T}$ with vector fields $\mathcal{A}_0$, $\mathcal{A}_1$ given by
\begin{align} \label{Df_vector_fields_A_0_A_1}
&\mathcal{A}_0 := \tau \cdot \nabla_\xi  - \nabla_\xi \Phi  \cdot \tau^{\bot} \, \frac{\partial}{\partial \alpha} ,~\mathcal{A}_1 := \sigma \, \frac{\partial}{\partial \alpha}.
\end{align}
Indeed, let $f \in C^\infty_c(\mathbb{R}^2 \times \mathbb{T})$. In particular, $\widetilde{f}:= f \circ \mathrm{P} \in C^\infty (\mathbb{R}^3)$. Furthermore, note
\begin{align*}
\nabla_\xi \widetilde{f} (\widetilde{p}) = \nabla_\xi f (p),~\frac{\partial \widetilde{f}}{\partial \widetilde{\alpha}}(\widetilde{p}) = \frac{\partial f}{\partial \alpha} (p),~\widetilde{p}=(\xi,\widetilde{\alpha}) \in \mathbb{R}^3,~p=\mathrm{P}(\widetilde{p}).
\end{align*} 
Thus the Stratonovich transformation rule (see~e.g.~\cite{Hsu02}) yields the claim since
\begin{align*}
& f(X_t) - f(X_0) = \widetilde{f}(\widetilde{X_t}) - \widetilde{f}(\widetilde{X_0}) \\
&= \int_0^t \Big( \tau(\widetilde{\alpha_s}) \cdot \nabla_\xi \widetilde{f}(\widetilde{X_s})  - \nabla \Phi (\xi_s) \cdot \tau^{\bot}(\widetilde{\alpha_s}) \, \frac{\partial \widetilde{f}}{\partial \widetilde{\alpha}} (\widetilde{X_s}) \Big) \, \mathrm{d}s +  \sigma \int_0^t \frac{\partial \widetilde{f}}{\partial \widetilde{\alpha}} (\widetilde{X_s}) \circ \mathrm{d}W_s \\
&=\int_0^t (\mathcal{A}_0 f) (X_s) \, \mathrm{d}s + \int_0^t (\mathcal{A}_1 f) (X_s) \circ \mathrm{d}W_s.
\end{align*}

Summarizing, Equation \eqref{eq_2D_fiber_lay_down_manifold} may reasonably be seen as a natural formulation of the two-dimensional fiber lay-down model on $\mathbb{R}^2 \times \mathbb{T}$.

\subsection{The equivalent two-dimensional model on \texorpdfstring{$\mathbb{R}^2 \times \mathbb{S}^1$}{}} \label{equivalentSDE_2d}  Since $\mathbb{T}$ and $\mathbb{S}^1$ are diffeomorphic with $\alpha \leftrightarrow v$, $v=\tau(\alpha)$, SDE \eqref{eq_2D_fiber_lay_down_manifold} can equivalently be formulated on the submanifold $\mathbb{R}^2 \times \mathbb{S}^1$ of $\mathbb{R}^4$. To obtain the equivalent SDE on $\mathbb{R}^2 \times \mathbb{S}^1$ we only have to compute the pushforward vector fields $\widetilde{\mathcal{A}_0}$, $\widetilde{\mathcal{A}_1}$ on $\mathbb{R}^2 \times \mathbb{S}^1$ associated to the vector fields $\mathcal{A}_0$ and $\mathcal{A}_1$ from \eqref{Df_vector_fields_A_0_A_1}. But first some notation: Each $x \in \mathbb{R}^4$ is written in the form $x=(\xi,v)  \in \mathbb{R}^2 \times \mathbb{R}^2$ and $v^\bot:=(-v_2,v_1)^T$. Now for $f \in C^\infty(\mathbb{R}^2 \times \mathbb{T})$ let $\widetilde{f} \in C^\infty(\mathbb{R}^2 \times \mathbb{S}^1)$ be given as 
\begin{align*}
\widetilde{f}(\widetilde{p}):=f(p),~p=(\xi,\alpha) \in \mathbb{R}^2 \times \mathbb{T},~\widetilde{p}=(\xi,v)=(\xi,\tau(\alpha)) \in \mathbb{R}^2 \times \mathbb{S}^1.
\end{align*}
So if $\widetilde{\frac{\partial}{\partial \alpha}}$ denotes the pushforward vector field of $\frac{\partial}{\partial \alpha}$ we get 
\begin{align*}
\widetilde{\frac{\partial}{\partial \alpha}} \widetilde{f} (\widetilde{p}) = \frac{\partial}{\partial \alpha} f(p) = \frac{\partial}{\partial \alpha} \widetilde{f}(\xi,\tau(\alpha)) = \tau^\bot(\alpha) \cdot \nabla_v \widetilde{f} \, (\widetilde{p}) = v^\bot \cdot \nabla_v \widetilde{f} \,(\widetilde{p})
\end{align*}
where we have used that any smooth function on the submanifold $\mathbb{R}^2 \times \mathbb{S}^1$ can locally be extended to some smooth function defined locally around $(\xi,v)$ in $\mathbb{R}^4$. Thus the vector field $\frac{\partial}{\partial \alpha}$ on $\mathbb{R}^2 \times \mathbb{T}$ corresponds to the vector field $v^\bot \cdot \nabla_v$ living on $\mathbb{R}^2 \times \mathbb{S}^1$. The latter is again canonically identified with $(0,v^\bot)$, in notation $v^\bot \cdot \nabla_v \equiv \,(0,v^\bot)$, $v \in \mathbb{S}^1$. Moreover, note that for $v \in \mathbb{S}^1$, the orthogonal projection $\Pi_{\mathbb{S}^1} [v] : \mathbb{R}^2 \rightarrow T_v \mathbb{S}^1$, $\Pi_{\mathbb{S}^1} [v] \,(y) = (I-v \otimes v)y$, is equal to $\Pi_{\mathbb{S}^1} [v] \,(y)= y \cdot v^\bot \, v^\bot$ since $T_v \mathbb{S}^1 = \text{span} \{ v^\bot \}$. Recall, $v \otimes v = v v^T$. Altogether, we get
\begin{align*}
&\widetilde{\mathcal{A}_0}= v \cdot \nabla_\xi - (I-v \otimes v) \nabla \Phi(\xi) \cdot \nabla_v   \equiv \, \begin{pmatrix} v \\ -(I-v \otimes v) \nabla \Phi (\xi) \end{pmatrix},\\
&\widetilde{\mathcal{A}_1} = \sigma \, v^\bot \cdot \nabla_v  \equiv \,\sigma \begin{pmatrix} 0 \\  v^\bot \end{pmatrix}
\end{align*}
and the equivalent Stratonovich SDE on $\mathbb{R}^2 \times \mathbb{S}^1$ associated to \eqref{eq_2D_fiber_lay_down_manifold} reads
\begin{align} \label{eq_2D_fiber_lay_down_manifold_equivalent}
\mathrm{d}X_t = \widetilde{\mathcal{A}_0} (X_t) \,\mathrm{dt} + \widetilde{\mathcal{A}_1} (X_t) \circ \mathrm{d}W_t
\end{align} 
where $W=\{W_t\}_{t \geq 0}$ is a standard one-dimensional Brownian motion.

\subsection{The \texorpdfstring{$d$}{}-dimensional model on \texorpdfstring{$\mathbb{R}^d \times \mathbb{S}^{d-1}$}{}} \label{isotropic_nd_model} In the following let $d \in \mathbb{N}$ with $d \geq 2$. Next, we translate the fiber lay-down model \eqref{eq_2D_fiber_lay_down_manifold_equivalent} in the most natural way to $\mathbb{R}^d \times \mathbb{S}^{d-1}$. The resulting equation is then called the \textit{$d$-dimensional or geometric fiber lay-down model}. But first note that the stochastic part of \eqref{eq_2D_fiber_lay_down_manifold_equivalent}, i.e., the term $\widetilde{\mathcal{A}_1} \circ \mathrm{d}W_t$, models Brownian motion (with diffusion constant $\sigma$) in the components representing $\mathbb{S}^1$. This holds since the vector field $\mathcal{V}:=v^\bot\, \equiv \, v^\bot \cdot \nabla_{\mathbb{R}^2} $, $v \in \mathbb{S}^1$, satisfies $ \mathcal{V}^2 = \Delta_{\mathbb{S}^1}$. Now there is another possibility of writing $\Delta_{\mathbb{S}^1}$ and generating a Brownian motion on $\mathbb{S}^1$. Therefore, set 
\begin{align*}
\mathcal{V}_j:= (I-v \otimes v)e_j \, \equiv \, (I-v \otimes v)e_j \cdot \nabla_{\mathbb{R}^2},~v \in \mathbb{S}^1,~j=1,2,
\end{align*}
and $e_j$ being the $j$-th unit vector. Then one easily verifies $ \sum_{j=1}^2 \mathcal{V}_j^2 =  \Delta_{\mathbb{S}^1}$. 
\medskip

Having the latter identity in mind, we introduce the $d$-dimensional model, $d \geq 2$, on the manifold $\mathbb{R}^d \times \mathbb{S}^{d-1}$ as
\begin{align} \label{eq_nD_fiber_lay_down_manifold}
\mathrm{d}X_t = \mathcal{N}_{0} (X_t) \,\mathrm{dt} + \sum_{j=1}^d \mathcal{N}_{j}  (X_t) \circ \mathrm{d}W^{(j)}_t
\end{align}
where $W=\{ W_t\}_{t \geq 0}$ is a standard $d$-dimensional Brownian motion and the vector fields $\mathcal{N}_{j}$ are defined  on $\mathbb{R}^d \times \mathbb{S}^{d-1}$ via
\begin{align} \label{Def_vector_fields_N_j}
\mathcal{N}_{0}= v \cdot \nabla_\xi - (I-v \otimes v) \nabla \Phi(\xi) \cdot \nabla_v \, \equiv \, \begin{pmatrix} v \\ -(I-v \otimes v) \nabla \Phi (\xi) \end{pmatrix},\\
\mathcal{N}_{j} = \sigma \, (I-v \otimes v)e_j \cdot \nabla_v \, \equiv \,\sigma \begin{pmatrix} 0 \\  (I-v \otimes v)e_j \end{pmatrix},~j=1,\ldots,d. \nonumber
\end{align}
Here $e_j$ is the $j$-th unit vector in $\mathbb{R}^d$ and any point from $\mathbb{R}^d \times \mathbb{S}^{d-1}$ is written in the form $(\xi,v)$. $\Phi \in C^\infty(\mathbb{R}^d)$ is only depending on $\xi$. In case $d=2$, SDE \eqref{eq_nD_fiber_lay_down_manifold} then slightly differs from the original equation derived in \eqref{eq_2D_fiber_lay_down_manifold_equivalent}. Nevertheless, their generators are the same. In particular, for $d=2$, any solution to \eqref{eq_2D_fiber_lay_down_manifold_equivalent} coincides weakly with any solution to \eqref{eq_nD_fiber_lay_down_manifold}. 
\medskip

Altogether, Equation \eqref{eq_nD_fiber_lay_down_manifold} may reasonably be seen as the natural $d$-dimensional version of the fiber lay-down model. The physical relevant scenario are considering the cases $d=2$ and $d=3$. This abstract SDE can now be embedded into $\mathbb{R}^{2d}$ by extending the vector fields $\mathcal{N}_j$ arbitrary to smooth vector fields on $\mathbb{R}^{2d}$. The extensions are still denoted by $\mathcal{N}_j$. The solution to \eqref{eq_nD_fiber_lay_down_manifold} is then obtained by solving the extended Stratonovich SDE in $\mathbb{R}^{2d}$, see \cite[Prop.~1.2.8]{Hsu02}. By choosing the trivial extensions, we may consider the following Stratonovich SDE in $\mathbb{R}^{2d}$ of the form
\begin{align} \label{eq_nD_fiber_lay_down_manifold_embedded} 
&\mathrm{d} \xi_t = v_t \,\mathrm{dt} \\
&\mathrm{d}v_t = -(I-v_t \otimes v_t) \,\nabla \Phi (\xi_t) \, \mathrm{dt} + \sigma \, (I-v_t \otimes v_t) \circ \mathrm{d}W_t. \nonumber
\end{align}
Its solution stays on $\mathbb{R}^d \times \mathbb{S}^{d-1}$ provided that the initial value lies on the manifold and gives a solution to \eqref{eq_nD_fiber_lay_down_manifold}, see \cite[Prop.~1.2.8]{Hsu02}. Here we define
\begin{align*}
(I-v_t \otimes v_t) \circ \mathrm{d}W_t:=\sum_{j=1}^d (I-v_t \otimes v_t)e_j \circ \mathrm{d}W_t^{(j)}.
\end{align*}

\subsection{The associated Kolmogorov operator} 
Let us calculate the Kolmogorov operator $L:C^\infty(\mathbb{R}^d \times \mathbb{S}^{d-1}) \rightarrow C^\infty(\mathbb{R}^d \times \mathbb{S}^{d-1})$ associated to \eqref{eq_nD_fiber_lay_down_manifold}. As in the two-dimensional case it holds $\sum_{j=1}^d \mathcal{N}_{j}^2 = \sigma^2 \Delta_{\mathbb{S}^{d-1}}$, see \cite[Theo.~3.1.4]{Hsu02}. Thus the associated  generator $L= \mathcal{N}_{0} + \frac{1}{2} \sum_{j=1}^d  \mathcal{N}_{j}^2$ is given by
\begin{align} \label{eq_generator_nd_model}
L&= v \cdot \nabla_\xi - (I-v \otimes v) \nabla \Phi(\xi) \cdot \nabla_v + \frac{\sigma^2}{2}  \, \Delta_{\mathbb{S}^{d-1}} \\
&= v \cdot \nabla_\xi - \text{grad}_{\mathbb{S}^{d-1}} V + \frac{\sigma^2}{2}  \,\Delta_{\mathbb{S}^{d-1}} \nonumber
\end{align}
with $V: \mathbb{R}^d \times \mathbb{S}^{d-1} \rightarrow \mathbb{R}$, $V(\xi,v):= \nabla \Phi(\xi) \cdot v$. Let us already remark that a stationary solution to the associated Fokker-Planck equation is explicitly known and is given up to normalization by $e^{-(d-1)\Phi}$, see Section \ref{section_mixing_properties}. Convergence to equilibrium of the process $X$, solving the $d$-dimensional fiber lay-down SDE, towards its stationary state is studied in Sections \ref{section_mixing_properties} and \ref{section_Hypocoercivity}.

\begin{Rm} \label{Rm_operator_replacement}
The generator $L:C^\infty(\mathbb{R}^2 \times \mathbb{T}) \rightarrow C^\infty(\mathbb{R}^2 \times \mathbb{T})$ of \eqref{eq_2D_fiber_lay_down_manifold} reads
\begin{align*}
L&= \tau \cdot \nabla_\xi  - \nabla \Phi (\xi)  \cdot \tau^{\bot} \, \frac{\partial}{\partial \alpha} + \frac{\sigma^2}{2}\, \frac{\partial^2}{\partial \alpha^2} =\tau \cdot \nabla_\xi - \text{grad}_{\mathbb{T}}\,V + \frac{\sigma^2}{2} \,\Delta_{\mathbb{T}}.
\end{align*}
Now $V:\mathbb{R}^2 \times \mathbb{T} \rightarrow \mathbb{R}$ is defined as $V(\xi,\alpha) = \nabla \Phi (\xi) \cdot \tau$. Thus we see that \eqref{eq_generator_nd_model} is the natural generalization of the two-dimensional fiber lay-down generator from $\mathbb{R}^2 \times \mathbb{T}$. So, in order to derive our $d$-dimensional equations, we could alternatively translate the generators to higher dimensions and search afterwards for a corresponding modeling SDE. Following this approach we end up again with the SDE \eqref{eq_nD_fiber_lay_down_manifold}.
\end{Rm}

\begin{figure}[tbp] 
\subfigure{\includegraphics[scale=0.47]{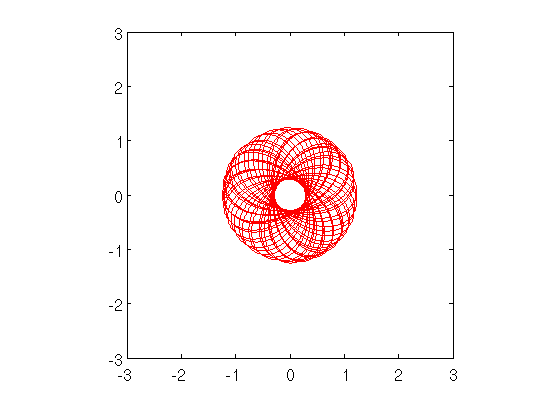}}
\subfigure{\includegraphics[scale=0.47]{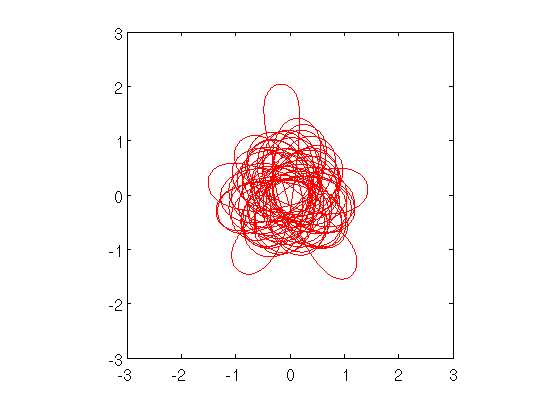}}
\subfigure{\includegraphics[scale=0.47]{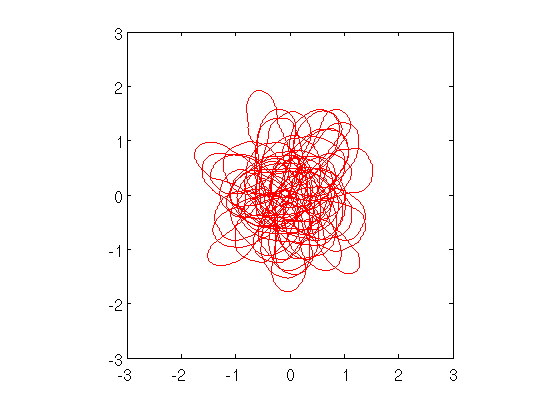}}
\subfigure{\includegraphics[scale=0.47]{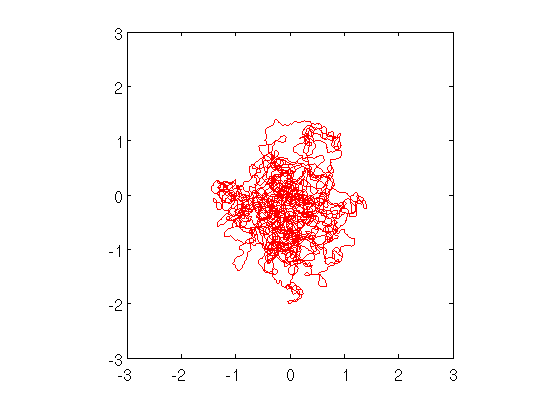}}
\caption{Two-dimensional case} \label{figure_2d}
\end{figure}

\begin{figure}[tbp] 
\subfigure{\includegraphics[scale=0.29]{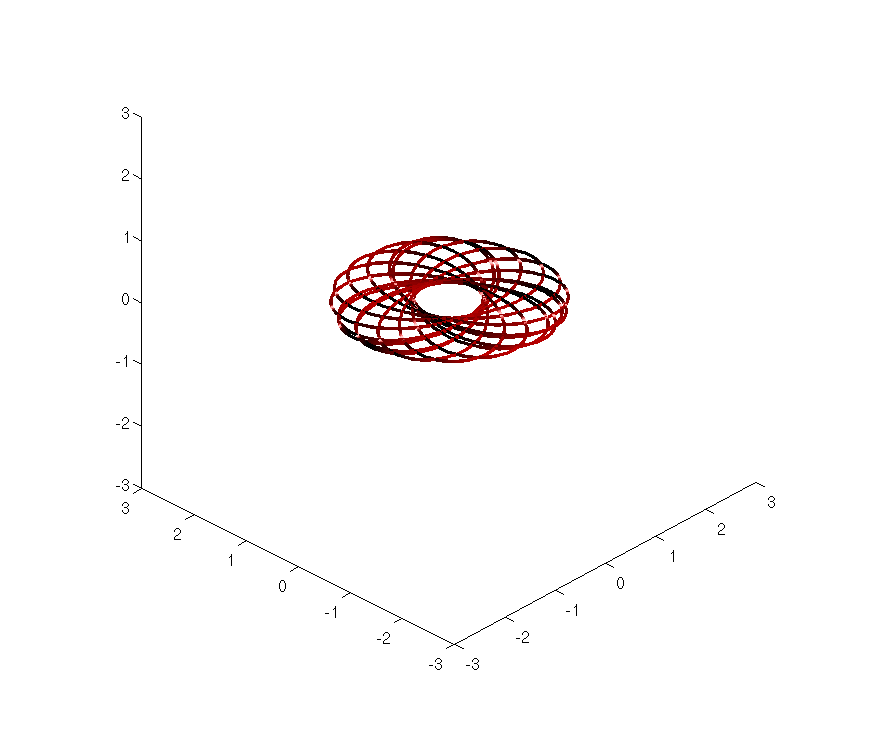}}
\subfigure{\includegraphics[scale=0.28]{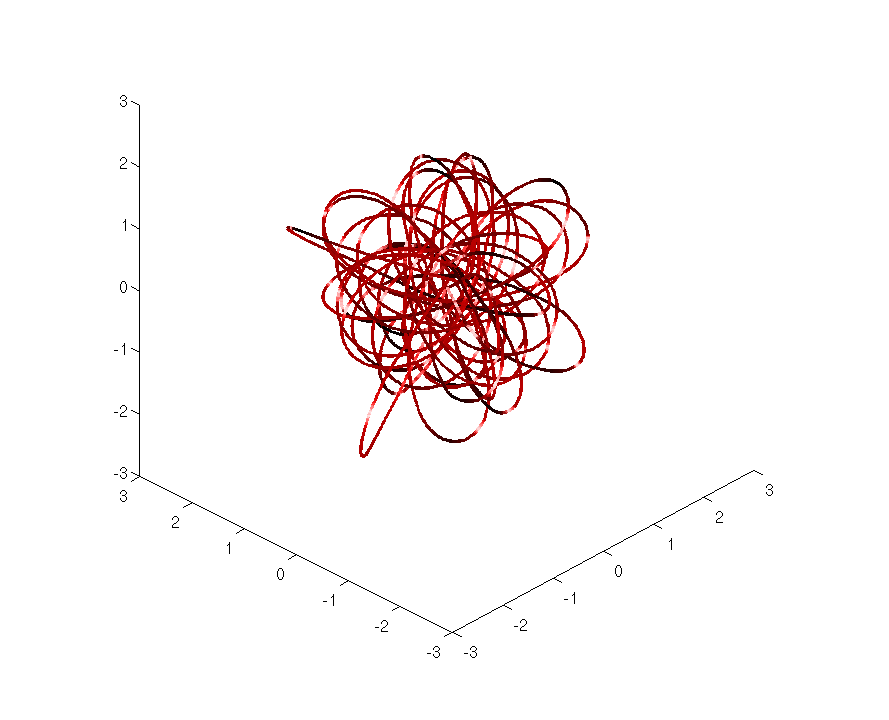}}
\subfigure{\includegraphics[scale=0.28]{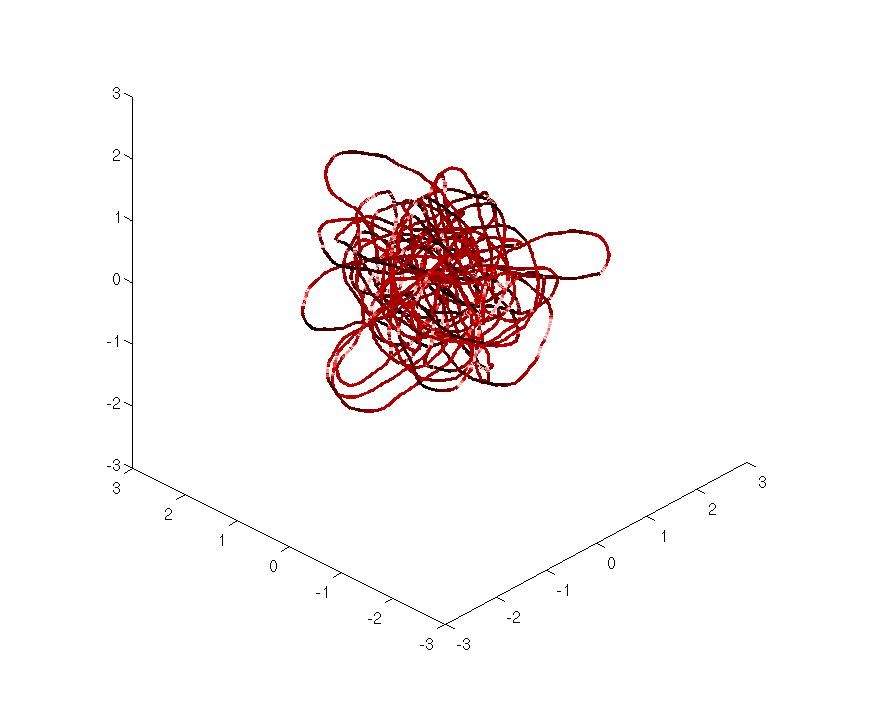}}
\subfigure{\includegraphics[scale=0.29]{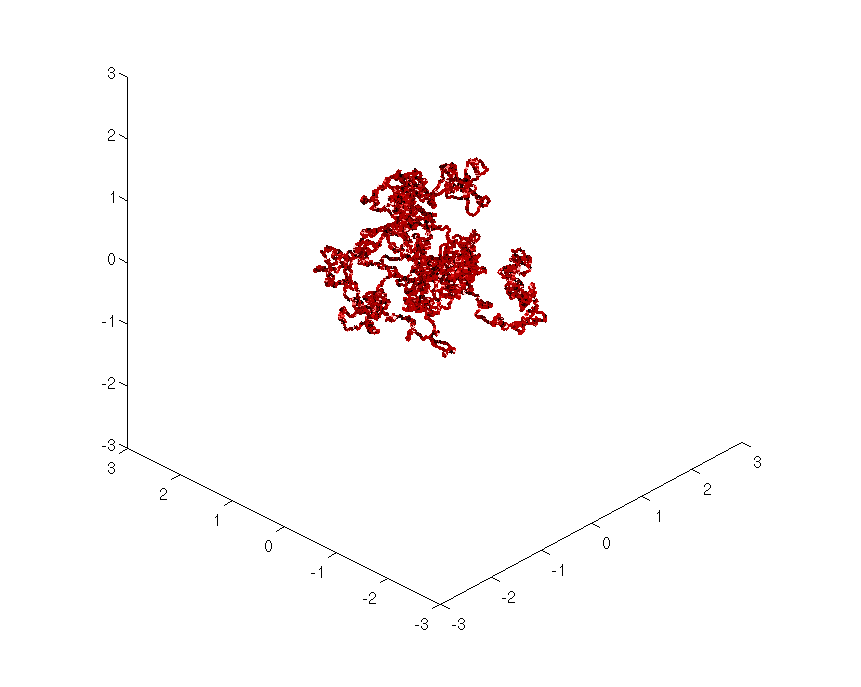}}
\caption{Three-dimensional case} \label{figure_3d}
\end{figure}

\subsection{Numerical simulations}

For simulating our previously defined $d$-dimensional fiber lay-down SDE \eqref{eq_nD_fiber_lay_down_manifold}, we may of course simulate SDE \eqref{eq_nD_fiber_lay_down_manifold_embedded} directly in the underlying euclidean space $\mathbb{R}^{2d}$. Nevertheless, this requires a consistent numerical algorithm staying on the manifold $\mathbb{R}^d \times \mathbb{S}^{d-1}$. 
\medskip

Instead, we take another approach and rewrite our fiber lay-down SDE in \textit{local coordinate form}. Therefore, let $d \in \mathbb{N}$, $d \geq 2$, and recall the following spherical coordinate system $\tau_{d-1}(\theta)=\tau_{d-1}(\theta_1,\ldots,\theta_{d-1})$ with $\theta_1 \in (0,2\pi)$ and $\theta_j \in (0,\pi)$ for $j=2,\ldots,d-1$, see the Appendix. Furthermore, note that $\tau_{d-1}$ is $2\pi$-periodic in $\theta_1$. Consider the following SDE
\begin{align} \label{SDE_local_parameter_form}
&\mathrm{d}\xi = \tau_{d-1} (\theta) \,\mathrm{dt} \\
&\mathrm{d}\theta_{j} = \left(  - \mathcal{G}_j(\theta) \, \nabla \Phi (\xi) \cdot n_j (\theta) +  \frac{\sigma^2}{2} \, \mathcal{G}_j^2(\theta) \, (j-1) \cot(\theta_j) \right)\, \mathrm{dt} + \sigma \,\mathcal{G}_j(\theta) \,\mathrm{d}W^{(j)}_t \nonumber
\end{align}
with $j=1,\ldots,d-1$. For abuse of notation, the time index $t$ is omitted. Here $\mathcal{G}_j$ (with $\mathcal{G}_{d-1}:=1$) and $n_j$ are given as
\begin{align*}
\mathcal{G}_j (\theta) = \prod_{i=j+1}^{d-1} \frac{1}{\sin(\theta_i)},~n_j (\theta)=  \frac{\partial_{\theta_j} \tau_{d-1}(\theta)}{| \partial_{\theta_j} \tau_{d-1}(\theta) |},~j=1,\ldots,d-1.
\end{align*}
and $W$ denotes a standard $(d-1)$-dimensional Brownian motion. Note that \eqref{SDE_local_parameter_form} writes as Stratonovich SDE in the same form. Let us justify this definition in the upcoming remark.

\begin{Rm} SDE \eqref{SDE_local_parameter_form} should be understood analogously as in Section \ref{2d_model_torus} as some manifold-valued Stratonovich SDE with state space $\mathbb{R}^d \times \mathbb{T} \times (0,\pi)^{d-2}$, or state space $\mathbb{R}^2 \times \mathbb{T}$ in case $d=2$, respectively. By using Formulas \eqref{formula_gradient_local} and \eqref{formula_Beltrami_local} from the Appendix, observe that the generator corresponding to  \eqref{SDE_local_parameter_form} coincides with the fiber lay-down generator $L$ computed on $\mathbb{R}^d \times \mathbb{T} \times (0,\pi)^{d-2}$. Now assume that there exists a solution $Y$ to \eqref{SDE_local_parameter_form} starting from $(\xi,\theta_1,\ldots,\theta_{d-1})$ and having infinite lifetime. Then it is easy to see that $X$, defined by $X_t:=\left(\xi(t),\tau_{d-1}(\theta_{1}(t),\ldots,\theta_{d-1}(t))\right)$, is a $L$-diffusion. Hence $X$ coincides weakly with any solution of the $d$-dimensional fiber lay-down SDE \eqref{eq_nD_fiber_lay_down_manifold} starting from $(\xi,\tau_{d-1}(\theta_1,\ldots,\theta_{d-1}))$. 
\end{Rm}

Summarizing, the simulation of SDE \eqref{SDE_local_parameter_form} gives us a (local) $L$-diffusion, $L$ being our $d$-dimensional fiber lay down generator. Thus it is reasonable to call \eqref{SDE_local_parameter_form} simply the \textit{$d$-dimensional fiber lay down SDE in local coordinate form}.  Note that in case $d=2$, \eqref{SDE_local_parameter_form} reduces to \eqref{eq_2D_fiber_lay_down} (or \eqref{eq_2D_fiber_lay_down_manifold}  respectively) and in case $d=3$ we obtain back the three-dimensional fiber lay down model derived in \cite{KMW12}. In Figure \ref{figure_2d} and Figure \ref{figure_3d}, let us illustrate and compare the $\xi$-trajectories in case $d=2$, $d=3$, for different values of $\sigma$ where $\sigma=0,~0.1,~0.5,~4.0$.  $\Phi$ is chosen as $\Phi=|\xi|^2$.

\section{Global solutions to the stochastic equations} \label{solutions_SDEs}

Before studying the long time behaviour of solutions to our $d$-dimensional fiber lay-down equations, we shall discuss basic existence statements of the underlying stochastic equations itself.

\subsection{The two-dimensional model} 

First we analyze the basic two-dimensional model, i.e., the It\^{o} SDE \eqref{eq_2D_fiber_lay_down} in $\mathbb{R}^3$. $\Phi(\xi)=a\,\xi_1^2 + b\,\xi_2^2$, $a,b \in \mathbb{R}$, $a,b \geq 0$, treats the physical relevant situation, see \cite{KMW09}. It is easy to verify that in this case the drift vector of the underlying SDE is not globally Lipschitz continuous and one may ask under which conditions the stochastic equation admits a global solution. Therefore, we rewrite the equation in a bit more general form as
\begin{align} \label{2d_model_general_form}
& \mathrm{d}\xi_t = \tau(\alpha_t)  \, \mathrm{dt} \\
& \mathrm{d}\alpha_t =  \Psi(\xi_t) \cdot \tau^{\bot}(\alpha_t) \, \mathrm{dt} + \sigma \, \mathrm{d}W_t. \nonumber
\end{align}

The following proposition shows that even under the assumption of $\Psi$ being locally Lipschitz, explosion is not possible and SDE \eqref{2d_model_general_form} admits a classical globally defined strong solution. 

\begin{Pp} \label{Pp_2d_model_general_form}
Let $\Psi: \mathbb{R}^2 \rightarrow \mathbb{R}^2$ be locally Lipschitz continuous. Given a filtered probability space $(\Omega, \mathcal{F}, \mathbb{P}, \{\mathcal{F}_t\}_{t \geq 0})$ together with an one-dimensional $\{ \mathcal{F}_t\}_{t \geq 0}$-Brownian motion $W=\{ W_t\}_{t \geq 0}$. Then for each $\mathcal{F}_0$-measurable $x_0: \Omega \rightarrow \mathbb{R}^3$ there exists a unique $\{ \mathcal{F}_t\}_{t \geq 0}$-adapted continuous $\mathbb{R}^3$-valued process $X=\{X_t\}_{t \geq 0}$ which satisfies $X_0=x_0$ $\mathbb{P}$-a.s.~and solves \eqref{2d_model_general_form} in integral form.
\end{Pp}

\begin{proof}
The classical Picard-Lindel"{o}f iteration scheme can be applied pathwise for each $ \omega \in \Omega$. This is due to the special structure of our equation. We follow the strategy from \cite[Ch.~5;~Ex.~2.19]{KS91}.  So let $t \geq 0$, $n \in \mathbb{N}_0=\mathbb{N} \cup \{0\}$, and define
\begin{align*}
X^{(0)}_t := x_0,~ X^{(n+1)}_t := x_0 + \int_{0}^{t}{b(X^{(n)}_s)}\,\mathrm{d}s + A\, W_t
\end{align*}
where $b(\xi,\alpha):=(\tau(\alpha),\Psi(\xi) \cdot \tau^{\bot}(\alpha) )^T,~A:= \left( 0,0,\sigma \right)^T$.
Then $X^{(n)}=\{ X^{(n)}_t\}_{t \geq 0}$, $n \in \mathbb{N}_0$, is easily seen to be continuous and $\{ \mathcal{F}_t \}_{t \geq 0}$-adapted.\\
Now fix $T > 0$ and $\omega \in \Omega$. Observe that the first two components of $X^{(n)}_t(\omega)$ are bounded by $N:=\left|x_0(\omega)\right| + T$ for each $n \in \mathbb{N}_0$, $t \in [0,T]$. Now $N$ only depends on $\omega$ and $T$. Thus we may restrict $b$ to $[-N,N]^{\,2} \times \mathbb{R}$ in the definition of $X^{(n)}_t(\omega)$ for all $n \in \mathbb{N}_0$, $t \in [0,T]$. Furthermore, observe that $b$ is Lipschitz continuous on each $K \times \mathbb{R}$, $K \subseteq \mathbb{R}^2$ compact. In particular, there exists some constant $L$, depending on $\omega$ and $T$, such that $\left|b(x) - b(y)\right| \leq L  \left|x-y\right|$ for all $x,y \in [-N,N]^{\,2} \times \mathbb{R}$. So altogether
\begin{align} \label{lipschitz_continuity_of_b^n}
\left|b(X^{(n)}_t(\omega)) - b(X^{(n-1)}_t(\omega))\right| \leq L \, \left|X^{(n)}_t(\omega)-X^{(n-1)}_t(\omega)\right| 
\end{align} 
holds for all $n \in \mathbb{N},~t \in [0,T]$. Now we may continue as in the classical case. Indeed, for those $n$ and $t$ we define 
\begin{align*}
D^{(n)}_t(\omega) = \max_{0 \leq s \leq t} \left|X^{(n)}_s(\omega) - X^{(n-1)}_s(\omega) \right|.
\end{align*}
Then \eqref{lipschitz_continuity_of_b^n} yields $D^{(n)}_t(\omega) \leq L \int_{0}^{t}{D^{(n-1)}_{s^\prime}(\omega) \,\mathrm{d}s^\prime}$ for $n \geq 2$ and $t \in [0,T]$. Moreover, $D^{(1)}_t(\omega)$ is bounded for $t \in [0,T]$ by $K:=T \, | b(x_0(\omega))| + \sigma \, \max_{0 \leq s \leq T} \left|W_s(\omega)\right|$ which also depends only on $\omega$ and $T$. Thus inductively we get
\begin{align*}
D^{(n)}_t(\omega) \leq  K \, \frac{  (L \,t)^{n-1}}{(n-1)!},~n \in \mathbb{N},~t \in [0,T].
\end{align*}
Now let $n > m$, $t \in [0,T]$. The latter estimate implies
\begin{align*}
\max_{0 \leq s \leq t} \left| X^{(n)}_s(\omega) - X^{(m)}_s(\omega) \right| &\leq \sum_{j=m}^{n-1}{\max_{0 \leq s \leq t} \left| X^{(j+1)}_s(\omega) - X^{(j)}_s(\omega) \right|} \leq K \sum_{j=m}^{\infty}{\frac{(L \, t)^{j}}{j!}}.
\end{align*}
Hence $X^{(n)}(\omega)$, $n \in \mathbb{N}$, converges uniformly on $[0,T]$ to some continuous function $X_t(\omega)$ where $0 \leq t \leq T$. Further note that $X_t(\omega)$ is contained in $[-N,N]^{\,2} \times \mathbb{R}$ for all $t \in [0,T]$, since $X^{(n)}_t(\omega)$ satisfies this property independent of $n \in \mathbb{N}$. Hence for $t \in [0,T]$ we conclude
\begin{align*}
\left| \int_{0}^{t}{b(X^{(n)}_s(\omega))\,\mathrm{d}s} - \int_{0}^{t}{b(X_s(\omega))\,\mathrm{d}s} \right|  \leq t\, L\, \max_{0 \leq s \leq t}\left| X^{(n)}_s(\omega) - X_s(\omega) \right|  \stackrel{n \rightarrow \infty}{\longrightarrow} 0.
\end{align*}
This shows that $X_t(\omega)$, $0 \leq t \leq T$, solves \eqref{2d_model_general_form} in integral form. Now the previous construction holds for all $T \geq 0$ and each $\omega \in \Omega$, hence the map $t \mapsto X_t(\omega)$,  $t \geq 0$, is well-defined and continuous. Clearly, $X_t$ is $\mathcal{F}_t$-measurable, $t \geq 0$. Concerning the uniqueness statement, see e.g.~\cite[Ch.~5,~Theo.~2.5]{KS91}.
\end{proof}

Now we are able to solve easily the two-dimensional fiber lay-down SDE \eqref{eq_2D_fiber_lay_down_manifold} with state space $\mathbb{R}^2 \times \mathbb{T}$ without using any abstract arguments. For simplicity we shall assume $\Phi \in C^\infty(\mathbb{R}^2)$ to stay consistent with the (smooth) manifold language. Concerning all the following uniqueness statements, we refer to \cite[Prop.~1.2.9]{Hsu02}.

\begin{Cor}
Let $\Phi \in C^\infty(\mathbb{R}^2)$. Then there exists a unique solution $X=\{X_t\}_{t \geq 0}$ with infinite lifetime of the two-dimensional fiber lay-down $\text{SDE}(\mathcal{A}_0,\mathcal{A}_1;W,x_0)$, see Equation \eqref{eq_2D_fiber_lay_down_manifold}, with state space $\mathbb{R}^2 \times \mathbb{T}$.
\end{Cor}

\begin{proof}
Consider the canonical projection $\mathrm{P}: \mathbb{R}^3 \rightarrow \mathbb{R}^2 \times \mathbb{T}$. Now we may choose some $\widetilde{x}_0: \Omega \rightarrow \mathbb{R}^3$ being $\mathcal{F}_0$-measurable that satisfies $x_0 = \mathrm{P} \circ \widetilde{x}_0$. Such an $\widetilde{x}_0$ exists since the map $\varphi:[0,2 \pi) \rightarrow \mathbb{T}, x \mapsto [x]$, is a Borel isomorphism, see \cite[Cor.~I.3.3]{Par67}. Proposition \ref{Pp_2d_model_general_form} (with $\Psi= -\nabla \Phi$) is applicable and yields the existence of a strong solution $\widetilde{X}$ to SDE \eqref{2d_model_general_form} with initial condition $\widetilde{x}_0$ relative to $(\Omega, \mathcal{F}, \mathbb{P}, \{\mathcal{F}_t\}_{t \geq 0}, W)$. The process $X$, defined by $X_t= \mathrm{P} \circ \widetilde{X}_t$, $t \geq 0$, satisfies $X_0=x_0$ $\mathbb{P}$-a.s.~and has infinite lifetime. The fact that $X$ solves $\text{SDE}(\mathcal{A}_0,\mathcal{A}_1;W,x_0)$ has already been verified in Section \ref{2d_model_torus}.
\end{proof}

\subsection{The \texorpdfstring{$d$}{}-dimensional fiber lay-down model} The same statement as before remains true for our general $d$-dimensional equation.

\begin{Pp} \label{Pp_solution_d_dimensional_equation}
Let $\Phi \in C^\infty(\mathbb{R}^d)$, $d \geq 2$. Then there exists a unique solution $X$ with infinite lifetime to the $d$-dimensional fiber lay-down $\text{SDE}(\mathcal{N}_{0},\ldots,\mathcal{N}_{n};W,x_0)$, see Equation \eqref{eq_nD_fiber_lay_down_manifold}, with state space $\mathbb{R}^d \times \mathbb{S}^{d-1}$.
\end{Pp}

\begin{proof}
By \cite[Theo.~1.2.9]{Hsu02} we know that there exists a unique continuous solution $X$ to $\text{SDE}(\mathcal{N}_{0},\ldots,\mathcal{N}_{n};W,x_0)$ with state space $\mathbb{R}^d \times \mathbb{S}^{d-1}$ up to its lifetime $e$. It remains to check that $e(X)=\infty$ holds $\mathbb{P}$-almost surely. Write $X_t=(\xi_t,v_t)$, $t < e(X)$. For each $f \in C_c^\infty(\mathbb{R}^d)$ we have
\begin{align*}
f(\xi_t) - f(\xi_0) = \int_0^t v_s \cdot \nabla_\xi f (\xi_s)\,\mathrm{d}s,~0 \leq t < e(X).
\end{align*}
In particular, this identity holds for all $f \in C_c^\infty(\mathbb{R}^d)$ which satisfy $f(\xi)=\xi_j$ for $j=1,\ldots,d$, inside some balls $B_r(0)$ of radius $r>0$ large enough. Then we easily conclude 
\begin{align} \label{eq_growth_xi}
\xi_t - \xi_0 = \int_0^t v_s\,\mathrm{d}s,~0 \leq t < e(X).
\end{align}
Consequently,  $\left|\xi_t-\xi_0\right| \leq t$ for $0 \leq t < e(X)$. Now assume that $e(X) < \infty$ holds on some $A \in \mathcal{F}$ with $\mathbb{P}(A) > 0$. But then $\left|X_t\right|_{\mathbb{R}^{2n}} \rightarrow \infty$ and hence $\left|\xi_t\right| \rightarrow \infty$ as $t \uparrow e(X)$ on $A$, see e.g.~\cite[Prop.~1.2.6]{Hsu02}. But this contradicts the previous inequality.
\end{proof}

\section{Convergence to equilibrium: Mixing properties} \label{section_mixing_properties}

As mentioned in the introduction, we start in this section with the investigating of the long time behaviour of our $d$-dimensional fiber lay-down process by proving strong mixing properties. We make use of a new version of Doob's theorem derived recently in \cite{GN12}. But first we introduce some notations that are used until the end of this article.

\subsection{Motivation and setup.} \label{subsection_Motivation_and_setup} In the following let $d \in \mathbb{N}$, $d \geq 2$. We focus attention on the $d$-dimensional fiber lay-down SDE with state space $\mathbb{M}:=\mathbb{R}^d \times \mathbb{S}^{d-1}$
\begin{align} \label{eq_nD_fiber_lay_down_manifold_2}
\mathrm{d}X_t = \mathcal{N}_{0} (X_t) \,\mathrm{dt} + \sum_{j=1}^d \mathcal{N}_{j}  (X_t) \circ \mathrm{d}W^{(j)}_t
\end{align}
where $W$ is a standard $d$-dimensional Brownian motion. For computational convenience, we replace $\Phi$ by $\frac{1}{d-1} \Phi$ in SDE \eqref{eq_nD_fiber_lay_down_manifold_2}. So $\mathcal{N}_0$ is redefined as 
\begin{align*} 
\mathcal{N}_0(\xi,v):=\begin{pmatrix} v \\ -\frac{1}{d-1}(I-v \otimes v) \nabla \Phi (\xi) \end{pmatrix}, (\xi,v) \in \mathbb{R}^d \times \mathbb{S}^{d-1}.
\end{align*}
The vector fields $\mathcal{N}_j$ are defined in \eqref{Def_vector_fields_N_j}. Convention: The first $d$-components of $\mathbb{R}^{2d}$ are abbreviated by $\xi$ whereas the last $d$-variables are denoted by $v$.  We consider potentials $\Phi:\mathbb{R}^d \rightarrow \mathbb{R}$ such that $e^{-\Phi} \in L^1(\mathbb{R}^d,\mathrm{d}\xi)$. Thus without loss of generality we assume in the following 
\begin{align*}
\int_{\mathbb{R}^d} e^{-\Phi}\,\mathrm{d}\xi =1.
\end{align*}
We introduce the probability measure $\mu$ on $(\mathbb{M},\mathcal{B}(\mathbb{M}))$ defined by 
\begin{align*}
\mu:= e^{-\Phi} d\xi \otimes \nu,~\nu:= \frac{1}{\text{vol}(\mathbb{S}^{d-1})} \, \mathcal{S}.
\end{align*}
Here $\mathcal{S}$ denotes the surface measure of $\mathbb{S}^{d-1}$, $\text{vol}(\mathbb{S}^{d-1})$ the surface area of $\mathbb{S}^{d-1}$ and $\mathcal{B}(\mathbb{M})$ the Borel-sigma-algebra on $\mathbb{M}$. The Kolmogorov operator $L=  v \cdot \nabla_\xi - \text{grad}_{\mathbb{S}^{d-1}} V + \frac{\sigma^2}{2}  \, \Delta_{\mathbb{S}^{d-1}}$, see \eqref {eq_generator_nd_model}, associated to SDE \eqref{eq_nD_fiber_lay_down_manifold} is also sometimes denoted by $L^{\text{K}}$ where $V: \mathbb{M} \rightarrow \mathbb{R}$ is now given as $V(\xi,v)= \frac{1}{d-1} \,\nabla \Phi(\xi) \cdot v$.
\medskip

Moreover, we denote by $X^x=\{ X^x_t\}_{t \geq 0}$ the solution to SDE \eqref{eq_nD_fiber_lay_down_manifold_2} with initial value $x \in \mathbb{M}$. Formally, its corresponding probability density $f_t$ (with respect to $\mathrm{d}\xi \otimes \mathcal{S}$ and Dirac initial state $x$) satisfies the Fokker-Planck equation $\partial_t f = L^{\text{FP}} f$ with $L^{\text{FP}}$ being the Fokker-Planck operator associated to the $d$-dimensional fiber lay-down model. The latter is given as the (algebraic) adjoint of $L^{\text{K}}$ in $L^2(\mathbb{M},\mathrm{d}\xi \otimes \mathcal{S})$. Thus, using Lemma \ref{Lm_adjoint_V} from the Appendix below, we get
\begin{align*}
L^{\text{FP}}= - v \cdot \nabla_\xi + \text{grad}_{\mathbb{S}^{d-1}} V -  \nabla \Phi (\xi) \cdot v + \frac{\sigma^2}{2}  \, \,\Delta_{\mathrm{S}^{d-1}}. 
\end{align*}
So, up to normalization, a stationary solution to the Fokker-Planck equation is now given by 
\begin{align*}
(\xi,v) \mapsto F(\xi,v) := e^{-\Phi(\xi)},~(\xi,v) \in \mathbb{M}.
\end{align*} 
Hence we expect that $f_t$ converges towards $\frac{1}{\text{vol}(\mathbb{S}^{d-1})} \, F$ as $t \rightarrow \infty$. In other words, $X_t^x$ should be distributed accordingly to $\mu$ for large values of $t \geq 0$. Now formally, $u(t,x)= \mathbb{E}[u_0(X^x_t)]$ solves the Kolmogorov PDE $\partial_t u = L u$ with initial state $u(t=0)=u_0$. Hence altogether the subsequent analysis in the previous and the upcoming section has the interpretation of proving convergence of $\mathbb{E}[u_0(X^x_t)]$ towards $\mathbb{E}[u_0(X^x_\infty)]:=\int_\mathbb{M} u_0 \,\mathrm{d}\mu$. We now start with the mentioned strong mixing properties.

\subsection{Mixing properties}

First, we need the following lemma. Therefore, for given vector fields $\mathcal{V}_1,\ldots,\mathcal{V}_r,$ on some manifold $\mathbb{X}$ we denote the least $\mathbb{R}$-vector space including all $\mathcal{V}_j$, $j=1,\ldots,r,$ which is closed under the Lie-bracket operation by $L(\mathcal{V}_1,\ldots,\mathcal{V}_r)$.

\begin{Lm} \label{Lm_Strong_Feller}
For the vector fields $\mathcal{N}_j$, $j=0,1,\ldots,d,$ from \eqref{Def_vector_fields_N_j} it holds
\begin{align*}
\mbox{dim} ~L\big(\mathcal{N}_1,\ldots,\mathcal{N}_d,\mathcal{N}_0 + \frac{\partial}{\partial t}\big)=2d \mbox{ at each point of $(0,\infty) \times \mathbb{M}$}
\end{align*}
where we assume $\sigma > 0$.
\end{Lm}

\begin{proof}
We may set $\sigma=1$ and choose $p=(t,\xi,v) \in (0,\infty) \times \mathbb{M}$ arbitrary. First note that 
\begin{align*}
\text{span} \{v\} \oplus \text{span} \{ (I-v \otimes v)e_j~|~j=1,\ldots,d \} = \text{span} \{v\} \oplus T_v\mathbb{S}^{d-1} =\mathbb{R}^d.
\end{align*} 
Observe that $\left[\mathcal{N}_j,\mathcal{N}_0 + \frac{\partial}{\partial t}\right]$ for $j=1,\ldots,d,$ is of the form
\begin{align*}
\mathcal{A}_j:=\left[\mathcal{N}_j,\mathcal{N}_0 + \frac{\partial}{\partial t}\right] = (I-v \otimes v) e_j \cdot \nabla_\xi + f_1^{\,(j)}(\xi,v) \cdot \nabla_v 
\end{align*}
for some smooth function $f_1^{\,(j)}: \mathbb{M} \rightarrow \mathbb{R}^d$. Hence, by the proof of Lemma \ref{Lm_formula_sphere_strong_feller} from the Appendix, we obtain 
\begin{align*}
\mathcal{A}:=\sum_{j=1}^d \left[\mathcal{N}_j,\left[\mathcal{N}_j,\mathcal{N}_0 + \frac{\partial}{\partial t}\right]\right] = -(d-1) v \cdot \nabla_{\xi} + f_2(\xi,v) \cdot \nabla_v.
\end{align*}
Here $f_2: \mathbb{M} \rightarrow \mathbb{R}^d$ is again smooth. So finally, under
\begin{align*}
\left(\mathcal{N}_0 + \frac{\partial}{\partial t}\right)(p),~\mathcal{N}_1(p),\ldots,~ \mathcal{N}_d(p),~\mathcal{A}_1(p),\ldots,~\mathcal{A}_d(p),~\mathcal{A}(p)
\end{align*}
we may always choose $2d$-linear independent vectors. The claim follows since the manifold $(0,\infty) \times \mathbb{M}$ has dimension $2d$.
\end{proof}

We need one more lemma. $X^x$ denotes the process solving \eqref{eq_nD_fiber_lay_down_manifold_2} with state space $\mathbb{M}$ defined on some underlying probability space $(\Omega, \mathcal{F} , \mathbb{P}, \{\mathcal{F}_t\}_{t \geq 0})$ satisfying the usual conditions. $B_r(z)$ denotes the open ball with radius $r>0$ centered at $z \in \mathbb{R}^d$.

\begin{Lm} \label{Lm_Blag}
Let $\Phi \in C^\infty(\mathbb{R}^d)$ and let $f \in C^\infty_c(\mathbb{M})$. The function $u(t,x):=\mathbb{E}[f(X_t^x)]$, $t \geq 0$, $x \in \mathbb{M}$, is continuously differentiable in $t$, twice continuously differentiable in $x$ and satisfies
\begin{align*}
\frac{\partial}{\partial t} u(t,x) = L u (t,x)
\end{align*} 
where $L:C^\infty(\mathbb{M}) \rightarrow C^\infty(\mathbb{M})$, $L= \mathcal{N}_0 + \frac{1}{2}\sum_{j=1}^d \mathcal{N}_j^2$.
\end{Lm}

\begin{proof}
This is a local statement, so let $x=(\xi^x,v^x) \in \mathbb{M}$ and let $T>0$ be arbitrary but fixed. Recall that the solution $X^y=\{ X^y_t\}_{t \geq 0}$, $y \in \mathbb{M}$, to the abstract equation \eqref{eq_nD_fiber_lay_down_manifold_2}, i.e., $\text{SDE}(\mathcal{N}_{0},\ldots,\mathcal{N}_{d};W,y)$, can be obtained as the solution to the embedded Stratonovich SDE in $\mathbb{R}^{2d}$ in which the vector fields $\mathcal{N}_j$ are extended to $\mathbb{R}^{2d}$ in the obvious way. Define $r_1=T$, $r_2=2T$, and choose some $\chi_1,\chi_2 \in C^\infty_c(\mathbb{R}^{d})$ with $\chi_1=1$ on $B_{r_2}(\xi^x)$, $\chi_2=1$ on  $B_2(0)$, $0 \leq \chi_1,\chi_2 \leq 1$ and define $\widetilde{\mathcal{N}_j}$ on $\mathbb{R}^{2d}$ as
\begin{align*}
\widetilde{\mathcal{N}}_j (\xi,v) = \chi_1(\xi) \,\chi_2(v) \,\mathcal{N}_j(\xi,v),~(\xi,v) \in \mathbb{R}^{2d},~j=0,\ldots,d.
\end{align*}
Now let $\widetilde{X}^y$ be the solution to the Stratonovich SDE
\begin{align} \label{SDE_widetildeX}
\mathrm{d} \widetilde{X}_t = \widetilde{\mathcal{N}}_0 \, \mathrm{dt} + \sum_{j=1}^d \widetilde{\mathcal{N}}_j \circ \mathrm{d}W_t^{(j)}
\end{align}
in $\mathbb{R}^{2d}$ with initial condition $\widetilde{X}_0=y \in \mathbb{M}$. All $\widetilde{\mathcal{N}}_j$ are still tangential to $\mathbb{M}$. Thus by \cite[Prop.~1.2.8]{Hsu02} we conclude that $\widetilde{X}^y$ stays on $\mathbb{M}$ and, as in the proof of Proposition \ref{Pp_solution_d_dimensional_equation}, $\widetilde{X}^y$ has infinite lifetime. Hence for each $y=(\xi^y,v^y) \in \mathbb{M}$ with $|\xi^y - \xi^x| < r_1$ it follows that $\widetilde{X}^y$ and $X^y$ both solve the Stratonovich SDE \eqref{SDE_widetildeX} in $\mathbb{R}^{2d}$ up to time $T$. Clearly, the latter SDE can be written in It\^{o}-form with global Lipschitz coefficients. Consequently, we have $\widetilde{X}_t^y=X^y_t$ \,$\mathbb{P}$-a.s.~for all $0 \leq t < T$ and all $y$ as specified above.\\
Now choose any extension of $f$ to some $\widetilde{f} \in C^\infty_c(\mathbb{R}^{2d})$. And since the equivalent It\^{o}-form of \eqref{SDE_widetildeX} has $C^\infty$-coefficients having compact support, we get that $\widetilde{u}(t,z):=\mathbb{E}[\widetilde{f}(\widetilde{X}_t^z)]$, $t \geq 0$, $z \in \mathbb{R}^{2d}$, is continuously differentiable in $t$, twice continuously differentiable in $z$ and satisfies
\begin{align*}
\frac{\partial}{\partial t} \widetilde{u}(t,z) = \widetilde{L} \widetilde{u} (t,z)
\end{align*} 
where $\widetilde{L}:C^\infty(\mathbb{R}^{2d}) \rightarrow C^\infty(\mathbb{R}^{2d}),~\widetilde{L}=\widetilde{\mathcal{N}}_0  + \sum_{j=1}^d \widetilde{\mathcal{N}}_j^{\,2}$, see  \cite[Prop.~2.7]{DPRO11} and \cite{BF61}. Hence the claim follows since $u(t,y)=\widetilde{u}(t,y)$ for $0 \leq t < T$, $y \in B_{r_1}(\xi^x) \times \mathbb{S}^{d-1}$, as well as since $\widetilde{L}h(y)=Lh(y)$ holds for each $h \in C^2(\mathbb{M})$ and all such $y$.
\end{proof}

We remark that under the assumption of the previous Lemma, $u(t,x)$ is even infinitely often differentiable in $x$, see \cite{BF61}. Nevertheless, we do not really need this stronger conclusion. Now we end up with our desired theorem. $B(\mathbb{M})$ (or $C(\mathbb{M})$ respectively) denotes the set of all Borel-measurable (or continuous respectively) real-valued functions on $\mathbb{M}$. The index $b$ denotes all bounded functions of the underlying set of functions. 

\begin{Thm} \label{Thm_strong_mixing}
Let $d \in \mathbb{N}$, $d \geq 2$, and let $\Phi \in C^\infty(\mathbb{R}^d)$ with $\int_{\mathbb{R}^d} e^{-\Phi}\,\mathrm{d}\xi=1$. Assume that $\sigma >0$. Then the $d$-dimensional fiber lay down process is strongly mixing, i.e., for all $f \in B_b(\mathbb{M})$ we have
\begin{align*}
\lim_{t \rightarrow \infty} \mathbb{E}[f(X_t^x)] = \int_\mathbb{M} f \,\mathrm{d}\mu
\end{align*}
uniformly in $x \in \mathbb{M}$ on compact subsets of $\mathbb{M}$. 
\end{Thm}

\begin{proof}
We apply a version of Doob's theorem, see \cite[Theo.~4.6]{GN12}. So we define operators $T_t$, $t \geq 0$, $T_t: B_b(\mathbb{M}) \rightarrow B_b(\mathbb{M})$ as $T_tf(x) = \mathbb{E}[f(X_t^x)]$ for $f \in B_b(\mathbb{M})$ and $x \in \mathbb{M}$. These operators are indeed well-defined. Therefore, let $X$ be the process defined on the path space $W(\mathbb{M})$ by $X_t(\omega)=\omega(t)$, $t\geq 0$, $\omega \in W(\mathbb{M})$. Then it holds $\mathbb{E}[f(X_t^x)] = \mathbb{E}_x[f(X_t)]$ where $\mathbb{E}_x$ denotes expectation with respect to $\mathbb{P}_x = \mathbb{P} \circ (X^x)^{-1}$. Thus for $A \in \mathcal{B}(\mathbb{M})$ and all $t \geq 0$ we have $T_t 1_A(x) = \mathbb{P}_x \{ X_t \in A \}$ which is $\mathcal{B}(\mathbb{M})$-measurable in the variable $x \in \mathbb{M}$, see \cite[Sec.~1.3]{Hsu02}. Hence $T_tf \in B_b(\mathbb{M})$ for all $f \in B_b(\mathbb{M})$. Furthermore, each $T_t$ is a linear, positive operator satisfying $\|T_tf\|_{\infty} \leq \|f\|_{\infty}$ for all $f \in B_b(\mathbb{M})$. Here $\|\cdot\|_{\infty}$ is the usual supremums norm. Moreover, for $A \in \mathcal{B}(\mathbb{M})$ and $x \in \mathbb{M}$ we have $T_t T_s 1_A (x) = T_{t+s} 1_A(x)$. This follows by the strong Markov property of the $L$-diffusion $X^x$, see \cite[Theo.~1.3.7]{Hsu02}. By applying monotone convergence we obtain $T_t T_s f = T_{t+s} f$ for each $f \in B_b(\mathbb{M})$, i.e., $\left(T_t\right)_{t \geq 0}$ is a semigroup.\\
Next let $f \in C^\infty_c(\mathbb{M})$. The function $u(t,x):=\mathbb{E}[f(X_t^x)]$, $t \geq 0$, $x \in \mathbb{M}$, is continuously differentiable in $t$, twice continuously differentiable in $x$ and satisfies
\begin{align*}
\frac{\partial}{\partial t} u(t,x) = L u (t,x),
\end{align*} 
see Lemma \ref{Lm_Blag}. Now choose $r>0$ such that the support of $f$ is contained in $B_r(0) \times \mathbb{S}^{d-1}$. By using Identity \eqref{eq_growth_xi} observe that for fixed $T\geq 0$, the function $u(t,\cdot)$ vanishes outside $K:=B_{r+T}(0) \times \mathbb{S}^{d-1}$ for $t \in [0,T]$.  Moreover, note that $\frac{\partial}{\partial t} u(t,x) = \mathbb{E}[Lf(X_t^x)]$ holds by It\^{o}'s formula for all $t \geq 0$. Thus 
\begin{align*}
\left|\frac{\partial}{\partial s} u(s,x)\right| \leq \sup_{y \in B_{r+2T}(0) \times \mathbb{S}^{d-1}} |Lf(y)| < \infty,~s \in [0,T],~x \in K.
\end{align*}
Note that $\frac{\partial}{\partial s} u(\cdot,\cdot)$ is $\mathcal{B}[0,t] \otimes \mathcal{B}(\mathbb{M})$-measurable for all $t \geq 0$. Now integrate the identity $u(t,\cdot) - f(\cdot) =\int_0^t \frac{\partial}{\partial s} u(s,\cdot) \, \mathrm{d}s$, $t \in [0,T]$, with respect to $\mu$. By the previous we may apply Fubini's theorem and get
\begin{align*}
\int_\mathbb{M} T_t f \, \mathrm{d}\mu - \int_\mathbb{M} f \, \mathrm{d}\mu = \int_0^t \int_\mathbb{M} L u(s,x)\, \mathrm{d}\mu(x) \,\mathrm{d}s =0,~t \in [0,T],
\end{align*}
where the last equality follows since $\int_\mathbb{M} L h \,\mathrm{d}\mu=0$ holds for all $h \in C^2_c(\mathbb{M})$, see Proposition \ref{Lm_adjoint_V} from the Appendix below. Hence we have $\int_\mathbb{M} T_t f \, d\mu = \int_\mathbb{M} f \, d\mu$ for all $f \in C_c^\infty(\mathbb{M})$. Thus by the functional monotone class argument, see \cite[Ch.~0,~Theo.~2.3]{BG68}, we easily conclude $\int_\mathbb{M} T_t f\, \mathrm{d}\mu = \int_\mathbb{M} f\, \mathrm{d}\mu$ for all $f \in B_b(\mathbb{M})$, i.e., $\mu$ is an invariant measure for $(T_t)_{t \geq 0}$.\\
Furthermore, Lemma \ref{Lm_Strong_Feller} just says that condition (E) from \cite{IK74} is satisfied. Hence by \cite[Theo.~3]{IK74} we obtain that there exists $p_t(x,y) \in C^\infty((0,\infty) \times \mathbb{M} \times \mathbb{M})$ with $\mathbb{P}\{X_t^x \in \mathrm{d}y \}=p_t(x,y)\, \mathrm{d}y$, $t > 0$. And since $X_t^x$ has infinite lifetime, i.e., $T_t 1_\mathbb{M}= 1_\mathbb{M}$ holds for all $t \geq 0$, the existence of such a function $p$ then implies by \cite[Lem.~5.1]{IK74} that $(T_t)_{t\geq 0}$ is already strongly Feller. Here strongly Feller means $T_t B_b(\mathbb{M}) \subset C_b(\mathbb{M})$, $t > 0$.\\
Next note that condition (ii) of Propositon 4.3 from \cite{GN12} is obviously satisfied since $\mu$ is stricly positive and $1_\mathbb{M}$ plays the role of the element $e$ in \cite[Prop.~4.3]{GN12}. Together with our previous conditions this already implies the weak irreducibiliy assumption from \cite{GN12}.\\
It is left to verify the strong continuity assumption in $L^1(\mathbb{M},\mu)$, i.e., condition (4.1) in \cite[Theo.~4.6]{GN12}. To verify the latter, it suffices to show that $\lim_{t \rightarrow 0}T_tf(x)=f(x)$ holds for all $f \in C_b(\mathbb{M})$ and all $x \in \mathbb{M}$, see \cite[Rem.~4.7]{GN12}. But this is obvious, use Lebesgue's dominated convergence.\\
Finally, observe that also condition (iv) of \cite[Theo.~3.4]{GN12} clearly holds, hence the claim follows by the generalized version of Doob's theorem, see \cite[Theo.~4.6]{GN12}. 
\end{proof}

\begin{Rm}
We remark that the original version of Doob's theorem, see \cite{DZ96} is not applicable in case of the fiber lay-down process. This is because the strong irreducibility assumption, i.e., the $t_0$-irreducibility property, is not satisfied in our case. The latter means that for some $t_0 >0$ it holds $T_{t_0} 1_\Gamma (x) > 0$ for all $x \in \mathbb{M}$ and each non-empty open $\Gamma \subset \mathbb{M}$. But we have already seen in the previous proof that such a property can't be satisfied simply due to the fact that the velocity vector of our fiber lay-down process lives on the sphere.
\end{Rm}

\section{Convergence to equilibrium: Hypocoercivity} \label{section_Hypocoercivity}

Now we switch to functional analysis. This section is independent from the stochastic ergodic results and statements obtained in the sections before. Now we apply the fascinating hypocoercivity theory derived in \cite{DMS11} and generalize \cite{DKMS11} to our higher-dimensional setting. For $d=2$, Theorem \ref{Thm_Hypocoercivity} below contains the result from \cite{DKMS11}. In contrast to \cite{DKMS11} we stay in the \textit{Kolmogorov picture}. In this way we obtain a nice comparision with \cite{GK08}. However, due to the underlying $L^2$-framework, the Kolmogorov setting is equivalent to the Fokker-Planck setting, see Remark \ref{Rm_equivalence_Kolmogorov_FP_picture}. Hence we may switch between both. 

\begin{Rm}
We remark that we do not specify any domains of all considered operators in this section. So our subsequent analysis in this section is done at first algebraically.
\end{Rm}

\subsection{Hypocoercivity}

We closely follow \cite[Sec.~1.3]{DMS11} and \cite{DKMS11}. Remember the general setup and notations introduced in Section \ref{subsection_Motivation_and_setup}. A convenient choice for the underlying Hilbert space is 
\begin{align*}
L^2(\mu)=L^2(\mathbb{M}, \mu),~\left( g, h \right)_{L^2(\mu)} = \int_M g h \, \mathrm{d}\mu,
\end{align*}
with $\mathbb{M}=\mathbb{R}^d \times \mathbb{S}^{d-1}$, $d \geq 2$, and $\mu$ as in Section \ref{subsection_Motivation_and_setup}. For further motivation on this choice, the interested reader may also consult e.g.~\cite{GK08} or \cite{CG08}. The corresponding norm on $L^2(\mu)$ is denoted by $\| \cdot \|_{L^2(\mu)}$. The Kolmogorov operator $L^{\text{K}}=L$ associated to our $d$-dimensional fiber lay-down SDE is written in the form $L=S-A$ where
\begin{align*}
S=\frac{1}{2} \sigma^2 \, \Delta_{\mathbb{S}^{d-1}},~A= - v \cdot \nabla_\xi + \text{grad}_{\mathbb{S}^{d-1}} V,~V(\xi,v)= \frac{1}{d-1} \,\nabla \Phi(\xi) \cdot v.
\end{align*}
Now Lemma \ref{Lm_adjoint_V} from the Appendix and the choice of $\mu$ implies that $A$ is antisymmetric on $L^2(\mu)$, $S$ is symmetric and negative semi-definite on $L^2(\mu)$ and we have
\begin{align} \label{eq_invariant_measure}
\int_\mathbb{M} Lu \, \mathrm{d}\mu = 0.
\end{align} 
Finally, in the following, the Kolmogorov PDE 
\begin{align} \label{Kolmogorov_PDE}
\partial_t u = L u= Su - Au
\end{align}
is considered as an abstract Cauchy problem in $L^2(\mu)$. Its solution subject to the initial condition $u_0 \in L^2(\mu)$ is denoted by $u(t)$, $t \geq 0$. Motivated by Section \ref{subsection_Motivation_and_setup} we have to study exponential convergence of $u(t)$ towards $\left( u_0,1 \right)_{L^2(\mu)}=\int_\mathbb{M} u_0 \,\mathrm{d}\mu$ in $L^2(\mu)$ as $t \rightarrow \infty$. We just remark that $u(t)$, $t \geq 0$, when starting with \eqref{Kolmogorov_PDE}, can really be given a rigouros meaning as some $\mathbb{E}^x[u_0(X_t)]$. This fact is justified by the theory of (generalized) Dirichlet forms, see e.g.~\cite{Tru05}.

\begin{Rm} \label{Rm_equivalence_Kolmogorov_FP_picture}
$ $
\begin{enumerate}
\item[(i)]
Let us explain the equivalence between the Kolmogorov and the Fokker-Planck $L^2$-setting. Following \cite{DKMS11}, the underlying Hilbert space in the Fokker-Planck picture is just $L^2(\widetilde{\mu})=L^2(\mathbb{M}, \widetilde{\mu})$ with $\widetilde{\mu}=e^{\Phi} \mathrm{d}\xi \otimes \nu$. Now consider the Hilbert space isomorphism 
\begin{align*}
T: L^2(\mu) \rightarrow L^2(\widetilde{\mu}),~h \mapsto Th,~Th(\xi,v) := e^{-\Phi(\xi)} h(\xi,-v).
\end{align*}
One readily checks that (algebraically) it holds $T L^{\text{K}} T^{-1} = L^{\text{FP}}$. Hence $u(t),~t \geq 0,$ solves the abstract (Kolmogorov) Cauchy problem $\partial_t u = L^{\text{K}}u$ in $L^2(\mu)$ with initial condition $u(t=0)=u_0$ if and only if $f(t),~t \geq 0,$ solves the abstract (Fokker-Planck) Cauchy problem $\partial_t f = L^{\text{FP}}f$ in $L^2(\widetilde{\mu})$ with initial condition $f(t=0)=f_0$. Here $f(t):=T u(t)$, $t \geq 0$. Moreover, note that
\begin{align*}
\left\| u(t) - \left( u_0, 1 \right)_{L^2(\mu)} \right\|_{L^2(\mu)} = \left\| f(t) - F \left( f_0, F \right)_{L^2(\widetilde{\mu})} \right\|_{L^2(\widetilde{\mu})},
\end{align*}
i.e., we have $\left\| u(t) - \int_\mathbb{M} u_0 \, \mathrm{d}\mu \, \right\|_{L^2(\mu)} = \left\| f(t) - F \int_\mathbb{M} f_0 \,\mathrm{d}\xi \mathrm{d}\nu \right\|_{L^2(\widetilde{\mu})}$. This shows that we may equivalently switch to the Fokker-Planck $L^2$-setting in Theorem \ref{Thm_Hypocoercivity} below.
\item[(ii)]
By integrating the identity $u(t) - u_0 = \int_0^t Lu(s) \, \mathrm{d}s$ with respect to $\mu$ and using \eqref{eq_invariant_measure} we obtain 
\begin{align} \label{eq_mu_invariant_measure}
\left( u(t), 1 \right)_{L^2(\mu)} = \int_\mathbb{M} u(t) \, \mathrm{d}\mu = \int_\mathbb{M} u_0 \, \mathrm{d}\mu= \left( u_0, 1 \right)_{L^2(\mu)}.
\end{align}
This has the interpretation that $\mu$ is an invariant measure. On the other hand \eqref{eq_mu_invariant_measure} implies $\left( f(t), F \right)_{L^2(\widetilde{\mu})} = \left( f_0, F \right)_{L^2(\widetilde{\mu})}$, that is, $\int_\mathbb{M} f(t) \, \mathrm{d}\xi \mathrm{d}\nu = \int_\mathbb{M} f_0 \, \mathrm{d}\xi \mathrm{d}\nu$. So in the Fokker-Planck $L^2$-setting, this just means that mass stays conserved.
\end{enumerate}
\end{Rm}

Next we fix the conditions on the potential $\Phi$ analogously to \cite{DKMS11}. Let $d \in \mathbb{N}$, $d \geq 2$. We denote by $H^{k}(e^{-\Phi} \mathrm{d}\xi)$, $k \in \mathbb{N}$, the space of all $k$-times weakly differentiable functions $h:\mathbb{R}^d \rightarrow \mathbb{R}$ whose derivatives up to order $k$ (including the function itself) are elements of $L^2(e^{-\Phi}\mathrm{d}\xi)$. Further, $\nabla^2_\xi$ or $\nabla^2$ stands for the Hessian matrix in $\mathbb{R}^d$.
\begin{enumerate}
\item[(H1)]
\textit{Regularity:} $\Phi \in W^{2,\infty}_{\text{loc}}(\mathbb{R}^d)$.
\item[(H2)]
\textit{Normalization:} $\int_{\mathbb{R}^d} e^{-\Phi}\,\mathrm{d}\xi =1$.
\item[(H3)]
\textit{Spectral gap condition:} There exists a positive constant $\Lambda < \infty$ such that
\begin{align*}
\int_{\mathbb{R}^d} \left| \nabla u \right|^2  e^{-\Phi}  \, \mathrm{d}\xi \geq \Lambda \int_{\mathbb{R}^d} u^2 e^{-\Phi} \, \mathrm{d}\xi
\end{align*}
for all $u \in H^{1}(e^{-\Phi} \mathrm{d}\xi)$ with $\int_{\mathbb{R}^d} u \,e^{-\Phi} \, \mathrm{d}\xi=0$.
\item[(H4)]
\textit{Pointwise condition:} There exists a constant $C < \infty$ such that
\begin{align*}
\left| \nabla^2 \Phi (\xi) \right| \leq C \left( 1+ \left| \nabla \Phi(\xi) \right|\right) \mbox{ for all } \xi \in \mathbb{R}^d.
\end{align*}
\end{enumerate}
 
\begin{Rm} \label{Rm_Elliptic_Regularity}
Let $d \in \mathbb{N}$, $d \geq 2$. In analogy to \cite{DKMS11} we need elliptic regularity estimates from \cite[Sec.~2]{DMS11}.  This $L^2 \rightarrow H^2$ regularity result requires the full strength of all conditions on $\Phi$. So assume (H1)-(H4) and let $f  \in L^2(e^{-\Phi} \mathrm{d}\xi)$. Assume that $u \in L^2(e^{-\Phi} \mathrm{d}\xi)$ with $\int_{\mathbb{R}^d} u \, e^{-\Phi} \, \mathrm{d}\xi=0$ solves the elliptic equation
\begin{align*}
- \frac{1}{d} \Delta u + \frac{1}{d} \nabla \Phi \cdot \nabla u + u =f.
\end{align*}
Then \cite[Prop.~5]{DMS11} together with \cite[Lem.~8]{DMS11} implies the estimates
\begin{align*}
&\|\nabla^2 u \|_{L^2(e^{-\Phi} \mathrm{d}\xi)} \leq C_1 \, \| f \|_{L^2(e^{-\Phi}\mathrm{d}\xi)} \\
&\| \left| W \right| \left|\nabla u\right| \|_{L^2(e^{-\Phi}\mathrm{d}\xi)} \leq C_2 \, \| f \|_{L^2(e^{-\Phi}\mathrm{d}\xi)}
\end{align*}
with $W=\sqrt{1+ \frac{1}{4 d} \left| \nabla \Phi \right|^2}$. Here $C_1, C_2 < \infty$ are constants independent of $f$ and $u$. More precisely, the elliptic equation from \cite[Sec.~2]{DMS11} is of the form $\omega_0^2 \, u - \nabla \cdot \left(  \omega_1^2 \, \nabla u \right) = \omega_0^2 \,f$. Note that the latter reduces for $\omega_0^2=e^{-\Phi}$ and $\omega_1^2=\frac{1}{d} e^{-\Phi}$ to the one from above. Now in order to apply \cite[Prop.~5,~Lem.~8]{DMS11}, Conditions (2.1), (2.4), (2.5), (2.6) and (2.7) from \cite{DMS11} must hold true. Indeed, consider \cite{DKMS11} and \cite[Sec.~3]{DMS11} for their verification.
\end{Rm}

After this preparation, we end up with the final theorem. 

\begin{Thm} \label{Thm_Hypocoercivity}
Let $d \in \mathbb{N}$, $d \geq 2$, and $\sigma > 0$. Assume conditions (H1)-(H4). Then, for every $\eta > 0$, the solution $u(t)$, $t \geq 0$, to the abstract Cauchy problem \eqref{Kolmogorov_PDE} in $L^2(\mu)$ with initial condition $u(t=0)=u_0 \in L^2(\mu)$ satisfies
\begin{align*}
\left\|u(t) - \left( u_0,1 \right)_{L^2(\mu)} \right\|_{L^2(\mu)}  \leq (1+ \eta) \left\| u_0 - \left( u_0,1 \right)_{L^2(\mu)} \right\|_{L^2(\mu)}   e^{-\lambda t}.
\end{align*}
Herein $\lambda$ is given by
\begin{align*}
\lambda=\frac{\eta}{1 + \eta} \, \frac{K_1 \, \sigma^2}{1+K_2 \, \sigma^2 + K_3 \, \sigma^4}
\end{align*}
and the constants $K_j < \infty$, $j=1, 2, 3$, are only depending on the potential $\Phi$.
\end{Thm}

\begin{proof}
\textit{Step 1:} In analogy to \cite{DKMS11} we start by realizing the program from \cite[Sec.~1.3]{DMS11}. Therefore, we choose the desired Hilbert space $\mathcal{H}$ as
\begin{align*}
\mathcal{H} = \left\{ g \in L^2(\mathbb{M},\mu) ~\Big|~ \left( g, 1 \right)_{L^2(\mu)} = \int_\mathbb{M} g \,\mathrm{d}\mu =0 \right\}.
\end{align*}
Next, we define the deviation $g(t)=u(t)- \left( u_0,1 \right)_{L^2(\mu)},~t \geq 0$. By $\eqref{eq_mu_invariant_measure}$ it follows that $g(t)$, $t \geq 0$, solves the abstract Cauchy problem \eqref{Kolmogorov_PDE} in the Hilbert space $\mathcal{H}$ subject to the initial condition $g(t=0)=u_0 - \left( u_0,1 \right)_{L^2(\mu)} \in \mathcal{H}$. Furthermore, denote the orthogonal projection to $\mathcal{N}(S)$ (the null space of $S$) by
\begin{align*}
\Pi g= \rho_g :=\int_{\mathbb{S}^{d-1}} g \, \mathrm{d}\nu.
\end{align*}
In particular, $\int_{\mathbb{R}^d} \rho_g \, e^{-\Phi} \, \mathrm{d}\xi=0$. Clearly, it holds $A \Pi g = - v \cdot \nabla_\xi \rho_g$ and hereby $\Pi A \Pi=0$. By \cite[Lem.~1]{DMS11}, the latter identity implies
\begin{align} \label{Boundedness_B}
\| Bg \|_{L^2(\mu)} \leq \frac{1}{2} \| (I - \Pi)g \|_{L^2(\mu)}
\end{align}
where $B$ is given as $B= (I + (A \Pi)^* A \Pi)^{-1} (A \Pi)^*$ and $I$ denotes the identity operator. For completeness, let us recapitulate the general strategy from \cite{DMS11} and \cite{DKMS11}.  First, the \textit{modified entropy functional} is defined as
\begin{align*}
H_\varepsilon[g]:= \frac{1}{2} \|g\|_{L^2(\mu)}^2 + \varepsilon \, \left( Bg,g \right)_{L^2(\mu)},~\varepsilon \in (0,1).
\end{align*}
Then $H_\varepsilon[g]$ is equivalent to $\|g\|_{L^2(\mu)}^2$, more precisely, \eqref{Boundedness_B} yields
\begin{align} \label{eq_equivalence_H}
\frac{1-\epsilon}{2} \|g\|_{L^2(\mu)}^2 \leq H_\varepsilon[g] \leq \frac{1+\epsilon}{2} \|g\|_{L^2(\mu)}^2.
\end{align}
Moreover, the evolution of $H_\varepsilon [g(t)]$ is given by $\frac{d}{dt} \,H_\varepsilon[g(t)] = -D_\varepsilon[g(t)]$ where $D_\varepsilon[g]$ is the \textit{entropy dissipation functional} 
\begin{align} \label{entropy_dissipation_functional} 
D_\varepsilon[g] = &-  \left( Sg,g \right)_{L^2(\mu)} + \varepsilon \left( BA \Pi g,g \right)_{L^2(\mu)} + \varepsilon \left( B A (I-\Pi)g,g \right)_{L^2(\mu)} \\~~~&- \varepsilon \left( A B g,g \right)_{L^2(\mu)} - \varepsilon  \left( BSg,g\right)_{L^2(\mu)}. \nonumber
\end{align}
The main step involves showing coercivity of $D_\varepsilon[g]$ for some appropriate $\varepsilon \in (0,1)$. Indeed, assume this to be true. Then there exists $\kappa > 0$ such that $D_\varepsilon[g] \geq \kappa \|g\|^2$ and hence by \eqref{eq_equivalence_H} we have $\frac{d}{dt} \,H_\varepsilon[g(t)] \leq -  \frac{2 \kappa}{1 + \varepsilon} H_\varepsilon[g(t)]$.  By using Gronwall`s lemma and again \eqref{eq_equivalence_H}, exponential convergence of $g(t)$ towards $0$ in $L^2(\mu)$ as $ t\rightarrow \infty$ follows. The desired rate of convergence is then finally be determined in step 4. \\
But first of all, we have to show coercivity $D_\varepsilon[\cdot]$ for some $\varepsilon \in (0,1)$ small enough.
\medskip

\textit{Step 2:} To do so, we start by verifying the \textit{microscopic and macroscopic coercivity} assumptions from \cite[Sec.~1.3]{DMS11}. The first one is satisfied due to the Poincar\'e inequality on $\mathbb{S}^{d-1}$
\begin{align*}
\frac{1}{d-1} \int_{\mathbb{S}^{d-1}}  \left( \text{grad}_{\mathbb{S}^{d-1}} h, \text{grad}_{\mathbb{S}^{d-1}} h \right)_{T\mathbb{S}^{d-1}}  \, \mathrm{d}\nu \geq \int_{\mathbb{S}^{d-1}}  h^2 \, \mathrm{d}\nu - \left( \int_{\mathbb{S}^{d-1}} h \, \mathrm{d}\nu \right)^2,  
\end{align*}
see \cite[Theo.~2]{Be89}. This yields
\begin{align*} 
- \left( S g, g \right)_{L^2(\mu)} \geq \frac{1}{2} \sigma^2 (d-1)  \left\| (I-\Pi)g \right\|_{L^2(\mu)}^2.
\end{align*}
Here $\left( \cdot, \cdot \right)_{T\mathbb{S}^{d-1}}$ is the usual Riemannian scalar product on the tangent bundle of the sphere. Using Lemma \ref{Gauss_formula_sphere} from the Appendix below, which is just a simple application of the Gaussian integral formula, and (H3) we conclude
\begin{align*}
\left\| A\Pi g \right\|_{L^2(\mu)}^2 &= \int_{\mathbb{R}^d} \int_{\mathbb{S}^{d-1}} (v \cdot \nabla_\xi \rho_g)^2 \,e^{-\Phi} \, \mathrm{d}\nu(v) \, \mathrm{d}\xi  \\
&= \frac{1}{d} \int_{\mathbb{R}^d} \left| \nabla_\xi \rho_g \right|^2 e^{-\Phi} \, \mathrm{d}\xi \geq   \frac{\Lambda}{d} \left\|  \Pi g \right\|_{L^2(\mu)}^2.
\end{align*}
The latter inequality is the desired macroscopic coercivity property. Henceforth, see \cite[Sec.~1.3]{DMS11}, we have
\begin{align*}
\left( BA \Pi g, g \right)_{L^2(\mu)} \geq \frac{\Lambda}{d+\Lambda} \left\| \Pi g\right\|_{L^2(\mu)}^2.
\end{align*}
So the sum of the first two terms in \eqref{entropy_dissipation_functional} is coercive. Now as described in \cite{DMS11} and \cite{DKMS11}, coercivity of $D_\varepsilon$ for $\varepsilon$ small enough follows if we can show that the operators $BA(I-\Pi)$, $AB$ and $BS$ are bounded and satisfy certain estimates. This is done next.
\medskip

\textit{Step 3:} Again we proceed as in \cite{DKMS11}. First of all, boundedness of $AB$ is automatically satisfied due to our previous conditions, see \cite[Lem.~1]{DMS11}. More precisely, we even have
\begin{align*} 
\left\| ABg \right\|_{L^2(\mu)} \leq \left\| (I-\Pi)g\right\|_{L^2(\mu)}.
\end{align*}
Furthermore, by using Formula \eqref {formula_n_sphere_2} from the Appendix, we conclude
\begin{align*}
(A\Pi)^*g=-\Pi Ag &= \nabla_\xi \cdot \int_{\mathbb{S}^{d-1}} v \, g\, \mathrm{d}\nu(v) - \int_{\mathbb{S}^{d-1}} \text{grad}_{\mathbb{S}^{d-1}} V(v) \,\mathrm{d}\nu(v) \\
&= \Big( \nabla_\xi - \nabla_\xi \Phi \Big) \cdot \int_{\mathbb{S}^{d-1}} v \,g \, \mathrm{d}\nu(v).
\end{align*}
This together with Equation \eqref{formula_eigenvalue_n_sphere} yields $(A\Pi)^*Sg = -(d-1) \frac{\sigma^2}{2} (A \Pi)^*g$. Consequently, $BS=-(d-1) \frac{\sigma^2}{2} B$ and therefore
\begin{align*} 
\left\|BSg\right\|_{L^2(\mu)} \leq (d-1) \frac{\sigma^2}{4} \left\|(I-\Pi)g\right\|_{L^2(\mu)}.
\end{align*}
Finally, boundedness of $BA$ is equivalent to boundedness of $(BA)^*$. To verify the latter, we make use of the elliptic regularity result mentioned before. First we calculate
\begin{align}\label{eq_calculation_hypo}
A^2 \Pi h &= A^2 \rho_h = - A ( v \cdot \nabla_\xi \rho_h) = \left( v, \nabla_\xi^2 \rho_h \, v \right)_{\text{euc}} -  \frac{1}{d-1} \left( \left(I-v\otimes v\right) \nabla_\xi \Phi, \nabla_\xi \rho_h \right)_{\text{euc}} \\
&= \left( v, \nabla_\xi^2  \rho_h \, v \right)_{\text{euc}} - \frac{1}{d-1} \left( \nabla_\xi \Phi, \nabla_\xi \rho_h \right)_{\text{euc}} + \frac{1}{d-1}\left( v, \nabla_\xi \Phi \right)_{\text{euc}} \left( v, \nabla_\xi \rho_h \right)_{\text{euc}}. \nonumber
\end{align}
Thus by the Gaussian integral formula, see Lemma \ref{Gauss_formula_sphere}, we get $\Pi A^2 \Pi h = \frac{1}{d} \Delta_\xi \rho_h - \frac{1}{d}  \nabla_\xi \Phi \cdot \nabla_\xi \rho_h$. For $g \in L^2(\mathbb{M},\mu)$, let $h=(I+(A\Pi)^*(A\Pi))^{-1}g$. The previous calculation yields
\begin{align} \label{elliptic_equation_proof_hypo}
\rho_g = \Pi h - \Pi A^2 \Pi h = - \frac{1}{d} \Delta_\xi \rho_h + \frac{1}{d}  \nabla_\xi \Phi \cdot \nabla_\xi \rho_h + \rho_h.
\end{align}
Note that $(BA)^*g=A^2 \Pi h$. Using \eqref{eq_calculation_hypo} and applying the elliptic regularity result from \cite{DMS11} (see Remark \ref{Rm_Elliptic_Regularity}) to Equation \eqref{elliptic_equation_proof_hypo}, we get
\begin{align*}
\|(BA)^*g\|_{L^2(\mu)}  &\leq \| |\nabla_\xi^2  \rho_h| \|_{L^2(e^{-\Phi}\mathrm{d}\xi)} + \frac{1}{d-1} \| |\nabla_\xi \Phi| |\nabla_\xi \rho_h| \|_{L^2(e^{-\Phi}\mathrm{d}\xi)} \\
&\leq C \, \|\rho_g\|_{L^2(e^{-\Phi}\mathrm{d}\xi)} \leq K \, \|g\|_{L^2(\mu)}
\end{align*}
for some constant $K < \infty$ independent of $g$. Thus also $\|BA g\|_{L^2(\mu)} \leq K \| g\|_{L^2(\mu)}$ and
\begin{align*} 
\|BA (I-\Pi) g\|_{L^2(\mu)} \leq K \|(I-\Pi) g\|_{L^2(\mu)}.
\end{align*}

\textit{Step 4:} By \cite[Sec.~1.3]{DMS11} and the previous steps we infer coercivity of $D_\varepsilon$ for some $\varepsilon \in (0,1)$ small enough. As explained at the end of step 1, this yields already exponential convergence.  It remains to verify the claimed rate of convergence. To do so, just copy the calculation from subsection 3.4 in \cite[Theo.~1]{DKMS11} via replacing $D$, $D:= \frac{\sigma^2}{2}$, through $\widetilde{D}:=(d-1)\frac{\sigma^2}{2}$ and $(2+\Lambda)$ through $d+\Lambda$ in the latter. Furthermore, note that in subsection 3.4 from \cite[Theo.~1]{DKMS11} the expression $1+C_2D^2$ can be replaced by some term of the form $1+C_3D +C_4D^2$ with $C_3, C_4$ being some finite constants not depending on $\sigma$. Having this in mind, the desired rate of convergence of the theorem is shown.
\end{proof}

\begin{Rm}
A solution to the (Kolmogorov of Fokker-Planck) Cauchy problem in case $\Phi \in C^\infty(\mathbb{R}^d)$ can easily be constructed using the well-known Heffer-Nier construction scheme, see \cite[Prop.~5.5]{HN05}. This result can then easily be generalized to the case of $\Phi$ being locally Lipschitz continuous (and bounded from below) as in \cite{CG10}. Details on this and further analytic approaches for proving convergence to equilibrium of the $d$-dimensional fiber lay-down process are discussed in a forthcoming publication of the first named and the last named author of the underlying article.
\end{Rm}

\section{Appendix} \label{Appendix}

For completeness, in this section we just prove some specific statements needed for our analysis in this paper. We start with the first one.

\begin{Lm} \label{Lm_formula_sphere_strong_feller}
Let $d \in \mathbb{N}$. $I_{\mathbb{S}^d}$ denotes the function $v \mapsto v$, $v \in \mathbb{S}^{d}$, where $\Delta_{\mathbb{S}^d}I_{\mathbb{S}^d}$ is understood componentwise. Then it holds
\begin{align} \label{formula_eigenvalue_n_sphere}
\Delta_{\mathbb{S}^d} I_{\mathbb{S}^d}= - d \,I_{\mathbb{S}^d},
\end{align}
\end{Lm}

\begin{proof}
We make use of a specific representation formula for the Laplace Beltrami. So let $\mathcal{V}_j$ be the vector field on $\mathbb{S}^d$ given by $v \mapsto \mathcal{V}_j(v)$, with $\mathcal{V}_j(v) \in T_v\mathbb{S}^d$ the $j$-th row of the matrix $I-v \otimes v$, $v \in \mathbb{S}^d$. It holds $\sum_{j=1}^{d+1} \mathcal{V}_j^2 = \Delta_{\mathbb{S}^d}$, see \cite[Theo.~3.1.4.]{Hsu02}. Choose $1 \leq j \leq d+1$. Now easily $\mathcal{V}_jI_{\mathbb{S}^d}(v)=\mathcal{V}_j(v)$. Thus the $i$-th component of $\mathcal{V}_jI_{\mathbb{S}^d}(v)$ is just $P_{ij}(v):= \delta_{ij} - v_iv_j$. We have
\begin{align*}
\mathcal{V}_j (P_{ij})(v)&= \sum_{n=1}^{d+1} \left(\delta_{nj}-v_nv_j\right) \frac{\partial}{\partial v_n} (- v_i v_j) \\
&= \begin{cases}
  v_i v_j^2 - (1-v_j^2)v_i = -v_i + 2 v_i v_j^2  & \text{if } i \not= j \\
  -2v_i(1-v_i^2) & \text{if } i=j 
\end{cases}
\end{align*}
and therefore 
\begin{align*}
\sum_{j=1}^{d+1} \mathcal{V}_j (P_{ij})(v) = \sum_{j=1,~j\not=i}^{d+1} \left(-v_i + 2 v_i v_j^2\right) -2v_i(1-v_i^2) = -d \,v_i.
\end{align*}
Hence we get $\Delta_{\mathbb{S}_d} I_{\mathbb{S}^d} = \sum_{j=1}^d \mathcal{V}_j^2\,I_{\mathbb{S}^d} = -d \,I_{\mathbb{S}^d}$.
\end{proof}

The following proposition is used for computing the Fokker-Planck operator corresponding to our higher-dimensional fiber lay-down process. $\mathcal{S}$ denotes the Riemannian measure on $(\mathbb{S}^d,\mathcal{B}(\mathbb{S}^d))$ and recall that $C^2(\mathbb{S}^d)$ is dense in $L^2(\mathbb{S}^d,\mathcal{S})$. Here $C^2$ means twice continuously differentiable.

\begin{Pp} \label{Lm_adjoint_V}
Let $d \in \mathbb{N}$. V: $\mathbb{S}^d \rightarrow \mathbb{R}$ is defined by $V(v)= \left( z , v \right)_{\text{euc}} $, where $z \in \mathbb{R}^{d+1}$ is fixed. Define the vector field $\mathcal{A}$ by $\mathcal{A}= \text{grad}_{\mathbb{S}^d} V$ and consider the operator $L= \mathcal{A}$ with domain $C^2(\mathbb{S}^d)$. Then its adjoint operator $L^*$ in $L^2(\mathbb{S}^d, \mathcal{S})$ is given on $C^2(\mathbb{S}^d)$ as
\begin{align*}
L^*= -\mathcal{A} + d  \left( z , v \right)_{\text{euc}} 
\end{align*}
\end{Pp}

\begin{proof}
Let $f,g \in C^2(\mathbb{S}^d)$. Clearly, $\int_{\mathbb{S}^d} \mathcal{A}(fg) \, \mathrm{d}\mathcal{S} = \int_{\mathbb{S}^d} \mathcal{A}(f) g \, \mathrm{d}\mathcal{S} + \int_{\mathbb{S}^d} f \mathcal{A}(g) \, \mathrm{d}\mathcal{S}$. So it suffices to show
\begin{align} \label{formula_n_sphere_2}
\int_{\mathbb{S}^d} \mathcal{A}(h) \, \mathrm{d}\mathcal{S} = d \int_{\mathbb{S}^d}  \left( z , v \right)_{\text{euc}}  h\, \mathrm{d} \mathcal{S}(v) ,~h \in C^2(\mathbb{S}^d).
\end{align}
So let $h \in C^2(\mathbb{S}^d)$. One has $\mathcal{A}(h) = \text{grad}_{\mathbb{S}^d}(V)(h) = \left(\text{grad}_{\mathbb{S}^d} V, \text{grad}_{\mathbb{S}^d} h\right)_{T\mathbb{S}^d}$ where $\left( \cdot, \cdot \right)_{T\mathbb{S}^d}$ denotes the canonically Riemannian scalar product on the tangent bundle $T\mathbb{S}^d$ of $\mathbb{S}^d$. Hence we get
\begin{align*} 
\int_{\mathbb{S}^d} \mathcal{A}(h) \, \mathrm{d}\mathcal{S} = \int_{\mathbb{S}^d} \left(\text{grad}_{\mathbb{S}^d} V, \text{grad}_{\mathbb{S}^d} h\right)_{T\mathbb{S}^d} \, \mathrm{d}\mathcal{S} = - \int_{\mathbb{S}^d} \left( \Delta_{\mathbb{S}^d}V \right) h \, \mathrm{d}\mathcal{S}.
\end{align*}
Finally, by Lemma \ref{Lm_formula_sphere_strong_feller} we have $\Delta_{\mathbb{S}^d}V = - d \left( z, I_{\mathbb{S}^d} \right)_{\text{euc}}$. Thus \eqref{formula_n_sphere_2} is shown and the claim follows.
\end{proof}

Moreover, we need one more lemma, which is just a simple application of the Gaussian integral formula.  $\nu$ denotes the normalized Riemannian measure of $\mathbb{S}^d$, i.e., $\nu=\frac{1}{\text{vol}(\mathbb{S}^d)} \mathcal{S}$ where $\text{vol}(\mathbb{S}^d)$ is the  surface area of $\mathbb{S}^d$.

\begin{Lm} \label{Gauss_formula_sphere}
Let $d \in \mathbb{N}$ and let $B$ be a matrix with entries $b_{ij} \in \mathbb{R}$, $i,j=1, \ldots, d+1$. Then 
\begin{align*}
\int_{\mathbb{S}^d} \left( Bv, v \right)_{\text{euc}} \,\mathrm{d}\nu(v) = \frac{1}{d+1} \, \sum_{j=1}^{d+1} b_{jj}.
\end{align*}
Hence for $z_1, z_2 \in \mathbb{R}^{d+1}$ we get $\int_{\mathbb{S}^d} \left(  z_1, v\right)_{\text{euc}} \left(  z_2, v \right)_{\text{euc}} \,\mathrm{d}\nu(v) = \frac{1}{d+1} \, \left( z_1, z_2 \right)_{\text{euc}}$.
\end{Lm}

\begin{proof}
Define $X:\mathbb{R}^{d+1} \rightarrow \mathbb{R}^{d+1}$ as $X(\xi)=B \,\xi$, $\xi \in \mathbb{R}^{d+1}$. The Gaussian integral formula then implies
\begin{align*}
\int_{\mathbb{S}^d}  \left( X(v) ,v \right)_{\text{euc}} \, \mathrm{d}\mathcal{S}(v) = \int_{B_1} \text{div}\, (X) (\xi) \, \mathrm{d}\xi = \text{vol}_{d+1}(B_1) \sum_{j=1}^{d+1} b_{jj}.
\end{align*}
Here $\text{vol}_{d+1}(B_1)$ denotes the Lebesgue volume of the $(d+1)$-dimensional unit ball $B_1$. By using the well-known relation $\text{vol}(\mathbb{S}^d)=(d+1) \, \text{vol}_{d+1}(B_1)$ the first claim follows.  For the second statement just set $B= z_2 \,z_1^T$.
\end{proof}

Finally, in order to do numerical simulations we have to compute some of our underlying objects in local coordinates. Therfore, recall the following spherical coordinate system $(U,x)$ given by $x= \tau_d^{-1}$ and $U= \text{Im}(\tau_d)$ with $\tau_{1}(\theta_1):= \left( \cos(\theta_1) \ \sin(\theta_1) \right)^T,~\theta_1 \in (0,2 \pi)$, and inductively for $d \in \mathbb{N}$, $d \geq 2$,
\begin{align*}
\tau_{d}:=\begin{pmatrix} \tau_{d-1}(\theta_1,\ldots,\theta_{d-1}) \sin(\theta_d) \\ \cos(\theta_d) \end{pmatrix},~\theta_d \in (0,\pi).
\end{align*}
Thus the Riemannian metric on $U$ is determined by $g^{(d)}_{ij}= \left( \partial_{\theta_i} \tau_d , \partial_{\theta_j} \tau_d \right)_{\text{euc}}$. So the density of the Riemannian volume measure $\sqrt{g}$ with $g=\text{det} \left( \left(g^{(d)}_{ij}\right) \right)$, denoted by $\varrho_d$, is given in this coordinate system by $\varrho_d = \prod_j^d | \partial_{\theta_j} \tau_d |$ since $g^{(d)}_{ij}=0$ for $i \not=j$. This yields the formula $\varrho_d = \varrho_{d-1} \sin^{d-1}(\theta_d),~d \geq 2$. In particular, 
\begin{align*}
\frac{\partial}{\partial \theta_j} \varrho_d = (j-1) \cot(\theta_j) \varrho_d,~d \in \mathbb{N},~j=1,\ldots,d.
\end{align*}
Define $n_j:=n_j^{(d)}:= {| \partial_{\theta_j} \tau_d |}^{-1}  \, \partial_{\theta_j} \tau_d,~j=1, \ldots, d$. Let $V$ be as in Lemma \ref{Lm_formula_sphere_strong_feller}. The gradient of $V$ is computed in local coordinates as $\text{grad}_{\mathbb{S}^d} V = \sum_{i,j=1}^{d} g^{ij} \frac{\partial V}{\partial x_j} \, \frac{\partial}{\partial x_i}$ where $(g^{ij})$ is the inverse matrix of $(g_{ij})$. Thus
\begin{align} \label{formula_gradient_local}
\text{grad}_{\mathbb{S}^d} V = \sum_{j=1}^d \mathcal{G}_j \left(z , n_j\right)_{\text{euc}} \frac{\partial}{\partial \theta_j}.
\end{align}
Here $\mathcal{G}_j:= {| \partial_{\theta_j} \tau_d |}^{-1} = \prod_{i=j+1}^d \frac{1}{\sin(\theta_i)}$ where the empty product in case $j=d$ is defined to be equal to $1$. Finally, for $\Delta_{\mathbb{S}^d}=\frac{1}{\sqrt{g}} \sum_{i,j=1}^d \frac{\partial}{\partial x_i} \left(\sqrt{g} g^{ij} \frac{\partial }{\partial x_j} \right)$ it holds
\begin{align} \label{formula_Beltrami_local}
\Delta_{\mathbb{S}^d} =\sum_{j=1}^d \mathcal{G}_j^2 \,\frac{\partial^2}{\partial \theta_j^2} + \sum_{j=1}^d \mathcal{G}_j^2 (j-1) \cot(\theta_j) \,\frac{\partial}{\partial \theta_j}.
\end{align}

\section*{Acknowledgement} This work has been supported by Bundesministerium f"{u}r Bildung und Forschung, Schwerpunkt \glqq Mathematik f"{u}r Innovationen in Industrie and Dienstleistungen\grqq , Projekt $03$MS$606$. The last named author thanks Benedict Baur, Florian Conrad, Benedikt Heinrich, Frank Seifried and Heinrich von Weizs"{a}cker for helpful discussions.

\end{document}